\documentclass[aop,preprint]{imsart}

\RequirePackage[OT1]{fontenc}
\RequirePackage{amsmath,amsfonts,amsthm,amssymb,color,mathrsfs}
\RequirePackage[numbers]{natbib}
\RequirePackage[colorlinks,citecolor=blue,urlcolor=blue]{hyperref}
\RequirePackage{enumitem}
\RequirePackage{comment}
\RequirePackage{stmaryrd}

\usepackage[T1]{fontenc}
\usepackage{enumitem}
\usepackage{amsmath,amsfonts,amsthm,amssymb,color,tikz, comment,csquotes}
\usepackage{mathrsfs}   
\usepackage{pdfsync}
\usepackage{hyperref}
\hypersetup{
  colorlinks   = true,
  citecolor    = blue,
  linkcolor    = blue
}
\numberwithin{equation}{section}

\newcommand{\di}{\diamond}

\newcommand{\dw}{\dot{W}}

\newcommand{\ii}{\imath}

\newcommand{\iot}{\int_{0}^{t}}

\newcommand{\ou}{[0,1]}
\newcommand{\1}{{\bf 1}}

\newcommand{\ts}{t^{*}}

\newcommand{\lastchange}[1]{{#1}}

\newcommand{\E}{\mathbb E}
\newcommand{\R}{\mathbb R}

\newcommand{\PP}{\mathbb P}

\newcommand{\be}{\mathbf{E}}

\newcommand{\bh}{\mathbf{H}}
\newcommand{\bp}{\mathbf{P}}

\newcommand{\ca}{\mathcal A}

\newcommand{\cf}{\mathcal F}
\newcommand{\cg}{\mathcal G}
\newcommand{\ch}{\mathcal H}

\newcommand{\cj}{\mathcal J}
\newcommand{\ck}{\mathcal K}

\newcommand{\cn}{\mathcal N}

\newcommand{\cs}{\mathcal S}

\newcommand{\al}{\alpha}

\newcommand{\ep}{\varepsilon}

\newcommand{\ga}{\gamma}

\newcommand{\ka}{\kappa}
\newcommand{\la}{\lambda}

\newcommand{\oom}{\Omega}

\newcommand{\vp}{\varphi}

\newcommand{\lp}{\left(}
\newcommand{\rp}{\right)}
\newcommand{\lc}{\left[}
\newcommand{\rc}{\right]}
\newcommand{\lcl}{\left\{}
\newcommand{\rcl}{\right\}}
\newcommand{\lln}{\left|}
\newcommand{\rrn}{\right|}
\newcommand{\lla}{\left\langle}
\newcommand{\rra}{\right\rangle}

\newtheorem{theorem}{Theorem}[section]

\newtheorem{definition}[theorem]{Definition}

\newtheorem{lemma}[theorem]{Lemma}

\newtheorem{proposition}[theorem]{Proposition}

\theoremstyle{remark}
\newtheorem{remark}[theorem]{Remark}
\theoremstyle{definition}
\newtheorem{notation}[theorem]{Notation}
\newtheorem{contribution}{Contribution}

\newcommand{\bean}{\begin{eqnarray*}}
\newcommand{\eean}{\end{eqnarray*}}
\newcommand{\ben}{\begin{enumerate}}
\newcommand{\een}{\end{enumerate}}
\newcommand{\beq}{\begin{equation}}
\newcommand{\eeq}{\end{equation}}

\begin{document}

\begin{frontmatter}
\title{Moment estimates for some renormalized\\ parabolic Anderson models}
%\runtitle{}
%\thankstext{T1}{Footnote to the title with the ``thankstext'' command.}

\begin{aug}
\author{\fnms{Xia} \snm{Chen}\thanksref{m1}\ead[label=e1]{xchen@math.utk.edu}},
\author{\fnms{Aur\'{e}lien} \snm{Deya}\thanksref{m2}\ead[label=e2]{Aurelien.Deya@univ-lorraine.fr}},
\author{\fnms{Cheng} \snm{Ouyang}\thanksref{m3}\ead[label=e3]{couyang@math.uic.edu}}
\and
\author{\fnms{Samy} \snm{Tindel}\thanksref{m4}
\ead[label=e4]{stindel@purdue.edu}
\ead[label=u1,url]{http://www.foo.com}}

%\thankstext{t1}{Some comment}
%\thankstext{t2}{First supporter of the project}
%\thankstext{t3}{Second supporter of the project}
\runauthor{X. Chen, A. Deya,  C. OUYANG AND S. TINDEL}

%\author{Xia Chen \and Aur\'{e}lien  Deya \and Cheng Ouyang \and Samy Tindel}

\affiliation{University of Tennessee Knoxville\thanksmark{m1}, University of Lorraine\thanksmark{m2} ,University of Illinois at Chicago\thanksmark{m3} and \\Purdue University\thanksmark{m4}}

\address{X. Chen\\Department of Mathematics\\University of Tennessee Knoxville\\
813 Swanston Street\\
TN 37996-1300\\
United States\\
\printead{e1}\\}
%\phantom{E-mail:\ }}

\address{A. Deya\\Institut Elie Cartan\\University of Lorraine\\B.P. 239, 54506 Vandoeuvre-l\`es-Nancy, Cedex\\France\\
\printead{e2}\\}
%\phantom{E-mail:\ }}

\address{C. Ouyang\\Department of Mathematics, Statistics\\ and Computer Science\\University of Illinois at Chicago\\
851 S Morgan St\\
Chicago  IL 60607\\United States\\
\printead{e3}\\}
%\phantom{E-mail:\ }}

\address{S. Tindel\\Department of Mathematics\\Purdue University\\
150 N. University Street\\
West Lafayette, IN 47907-2067\\United States\\
\printead{e4}\\}
%\phantom{E-mail:\ }}

\end{aug}

\begin{abstract}
The theory of regularity structures enables the definition of the following parabolic Anderson model in a very rough environment: $\partial_{t} u _{t}(x) = \frac12 \Delta u_{t}(x) + u_{t}(x) \, \dot W_{t}(x)$, for $t\in\R_{+}$ and $x\in \R^{d}$, where $\dot W_{t}(x)$ is a Gaussian noise whose space time covariance function is singular. In this rough context, we shall give some information about the moments of $u _{t}(x)$ when the stochastic heat equation is interpreted in the Skorohod as well as the Stratonovich sense. Of special interest is the critical case, for which one observes a blowup of moments for large times.
%\keywords{parabolic Anderson model \and regularity structures \and Skorohod equation \and Stratonovich equation \and moment estimate \and critical time}
%\subclass{60L30 \and 60L50 \and 60F10 \and 60K37}
\end{abstract}



\end{frontmatter}

\section{Introduction}
The parabolic Anderson model (sometimes abbreviated as \textsc{pam} in the sequel) is a linear partial differential equation in a random environment. As for other widely studied objects, many different versions of the model have been analyzed in the literature. In this paper we are concerned with the following continuous version of \textsc{pam} defined for $(t,x)\in\R_+\times\R^d$:
\begin{equation}\label{eq:she-intro}
\partial_{t} u _{t}(x) = \frac12 \Delta u_{t}(x) + u_{t}(x)\, \dot{W}_t(x),
\end{equation}
where $\Delta$ stands for the Laplace operator and $\dot{W}$ is a centered Gaussian noise. Equation \eqref{eq:she-intro} is obviously a stochastic PDE, and the generalized differential element $u_t(x)\dot{W}(x)$ will be interpreted either in the Skorohod sense (for which the product $u_t(x)\diamond\dot{W}_t(x)$ is considered as a Wick product) or in the Stratonovich sense (where the product $u_t(x)\dot{W}_t(x)$ is the usual one). Notice that we consider those two versions of the model for the sake of generality, but also because we will transfer some information from the Skorohod to the Stratonovich equation.

The study of moments for equation \eqref{eq:she-intro} is at the heart of Anderson's initial motivation in the model. Indeed, the moments of $u_t(x)$ characterize the so-called intermittency phenomenon, as described in e.g. \cite{Kh}. Moments are also related to the localization of eigenvectors for the Anderson operator $L=\frac{1}{2}\Delta+\dot{W}$, since $u_t(x)$ can be seen as the Laplace transform of the spectral measure of $L$ (see \cite[Relation (2.27)]{konig_book}). This is why quantities of the form $\be[|u_{t}(x)|^{p}]$, for a given $p>1$ and for the solution $u$ to \eqref{eq:she-intro}, have been intensively analyzed in the recent past. Let us mention \cite{CFK,CFJK,CJK} when $\dot{W}$ is a white noise in time and \cite{Ch14,CHNT,HHNT} for fractional noises.

The current paper can be seen as an additional step towards moment estimates for the parabolic Anderson model. Namely our study aims at giving some information about the moments of equation \eqref{eq:she-intro} when $\dot{W}$ is a very rough environment, given as the formal derivative of a multiparametric fractional Brownian motion $W$. Specifically, consider a centered Gaussian process $W$ indexed by $\R_{+}\times\R^{d}$ and defined on a complete probability space $(\Omega, \mathcal{F}, \mathbb{P})$, whose covariance is expressed as:
\begin{equation*}
\be\lc  W_{t}(x) \, W_{s}(y) \rc
=
R_{0}(s,t) \, \prod_{j=1}^{d} R_{j}(x_{j},y_{j}) ,
\end{equation*}
where $\{H_{j};\, 0\le j \le d\}$ is a family of Hurst indices in $(0,1)$ and the covariance function $R_{j}$ is defined by
\begin{equation}\label{eq:def-Rj}
R_{j}(u,v)
=
\frac12\lp |u|^{2H_{j}} + |v|^{2H_{j}} - |u-v|^{2H_{j}}  \rp, \qquad u,v\in\R.
\end{equation}
Then the noise $\dw$ driving equation \eqref{eq:she-intro} has to be thought of as the (ill-defined) derivative $\partial_{t\,x_{1}\cdots x_{d}}^{d+1}W_{t}(x)$. In this paper we are interested in noises which are rougher than white noise in at least some directions. Otherwise stated we wish some of the $H_j$'s in \eqref{eq:def-Rj} to be smaller than $\frac{1}{2}$. Recall that the covariance function of $\dot{W}$ is formally written as
\begin{equation}\label{eq:def-gamma}
\be\lc  \dw_{t}(x) \, \dw_{s}(y) \rc
=
\ga_{0}(t-s) \, \ga(x-y),
\quad\text{with}\quad
\ga(x-y)\equiv\prod_{j=1}^{d} \ga_{j}(y_{j}-x_{j}),
\end{equation}
where each $\ga_{j}$ is the distributional derivative $\partial_{uv}^{2}R_{j}$. Notice that whenever $H_{j}<\frac12$ the covariance $\ga_{j}$ is a distribution. Therefore we will often express $\ga_{0}$ and $\ga$ in Fourier modes as
\begin{equation}\label{eq:rep-gamma-fourier}
\ga_{0}(t) =  \int_{\R} e^{\ii \la t} \mu_{0}(d\la),
\quad\text{and}\quad
\ga(x) =  \int_{\R^{d}} e^{\ii \xi \cdot x} \mu(d\xi),
\end{equation}
where the measures $\mu_{0}$ and $\mu$ on $\R^{d}$ are respectively defined by 
\begin{equation}\label{eq:def-mu}
\mu_{0}(d\la)= c_{0} \, |\la|^{1-2H_{0}}d\lambda,
\quad\text{and}\quad
\mu(d\xi) = c_{\bh} \prod_{j=1}^{d} |\xi_{j}|^{1-2H_{j}}d\xi,
\end{equation}
where $\bh$ denotes the vector $(H_1,\ldots,H_d)$ and where $c_{0},c_{\bh}$ are explicit positive constants. We should already observe at this point that the mere existence of a solution to equation~\eqref{eq:she-intro} in the rough environment given by~\eqref{eq:def-gamma} requires a delicate analysis of intersection local times in the Skorohod setting \cite{Ch18}, and a cumbersome renormalization procedure for the Stratonovich case (see \cite{De16,De17,HL} for some related models).

\smallskip

In order to describe the main results contained in this article let us start with the Skorohod setting for equation \eqref{eq:she-intro}, for which we will assume that $H_0>\frac{1}{2}$. Within this framework we define a family of coefficients describing the behavior of our model. Namely set
\begin{equation}\label{defi:j-ast}
J_{*} = \lcl  1\le j \le d; \, H_{j} <\frac12  \rcl, 
\quad d_{*} = |J_{*}|,
\quad H_{*} = \sum_{j\in J_{*}} H_{j},
\quad H = \sum_{j=1}^{d} H_{j}.
\end{equation}
We also define some similar quantities $J^{*}, d^{*}, H^{*}$ for the indices such that $H_{j} \ge\frac12$. Then the reference \cite{Ch18} exhibits a subcritical regime, for which there is existence and uniqueness of the solution to \eqref{eq:she-intro} interpreted in the Skorohod sense. This subcritical regime is characterized by the following set of conditions on $J_*, d_*, H_*$:
\begin{equation}\label{eq:sko-subcritical-regime}
d-H<1,
\quad\text{and}\quad
4(1-H_{0}) + 2(d-H) + (d_{*}-2H_{*}) <4.
\end{equation}
Denoting by $u^\di$ this solution, it is also proved in \cite{Ch18} that $u^\di$ admits moments of all orders, namely,
$$
\be\big[|u_t^\di(x)|^p\big]<\infty,\quad\mathrm{for\ all}\ t\geq0, \, x\in\R^d, \mathrm{and}\ p\geq1.
$$
In contrast with this nice situation, in the current article we will focus on the so-called critical regime. This means that~\eqref{eq:sko-subcritical-regime} is replaced by the following condition on $J_*, d_*, H_*$:
\begin{equation}\label{eq:sko-critical-regime}
d-H=1,
\quad\text{and}\quad
4(1-H_{0})  + (d_{*}-2H_{*}) <2.
\end{equation}
Under condition \eqref{eq:sko-critical-regime}, the moments of \lastchange{$u^\di_t(x)$} blow up for large time. This assertion will be quantified precisely in our article, leading to our first contribution (which will be stated more rigorously in Theorem \ref{thm:explosion-critical}).

\begin{contribution}\label{contribution1}
Assume that condition \eqref{eq:sko-critical-regime} is met, and recall that $u^\di$ designates the Skorohod solution of equation \eqref{eq:she-intro}. For all $p>1$ we define a critical time $t_0(p)$ as
\begin{align}\label{intro-t0}
t_0(p)=\frac{C_{H_0,\bh}}{(p-1)^{1/(2H_0-1)}}
\end{align}
where $C_{H_0,\bh}$ is given by an explicit variational inequality. Then the following holds true:
\begin{enumerate}[wide, labelwidth=!, labelindent=0pt, label=(\roman*)]
\setlength\itemsep{.1in}
\item For any $p\geq 2$, if $t<t_0(p)$ we have $\E\big[|u_t^\di(x)|^p\big]<\infty$ for all $x\in\R^d$;
\item\label{cont:critical-p-less-2} 
For all $p>1$, if $t>t_0(p)$ the $p$-th moment of $u^\di_t(x)$ blows up.
\end{enumerate}
\end{contribution}

\begin{remark}\label{rmk: contri 1}
Our Contribution \ref{contribution1} gives a rather complete picture of the moments problem for the Skorohod equation \eqref{eq:she-intro} in the critical regime. In addition, it is also clear from equation~\eqref{intro-t0} that $p\mapsto t_0(p)$ decreases from $+\infty$ to $0$ as $p$ varies in $(1,\infty)$. We firmly believe that $t_0(p)$ separates well behaved from ill-behaved $p$-th moments for $u^\di$, and this is what our Contribution \ref{contribution1} asserts for $p\geq2$. However, item (ii) in Contribution \ref{contribution1} only yields an upper bound for the critical time $t_0(p)$ when $p\in(1,2)$. Hence a full characterization of $t_0(p)$ for $p<2$ is still an open problem.
\end{remark}

{
\begin{remark}
Note that Contribution  \ref{contribution1} covers the particular and important case in which $d=2$ and $H_0=1, H_1=H_2=\frac{1}{2}$, that is, $\dot{W}$ is the spatial white noise.  In this situation, the constant $C_{H_0,\bh}$ in \eqref{intro-t0} is given by $\kappa(2,2)^{-4}$ where $\kappa(2,2)$ is the Gagliardo-Nirenberg constant for $p=2, d=2$ (see Remark \ref{relation GN inequality} below for more details). 
\end{remark}
}

{Our second contribution focuses on equation \eqref{eq:she-intro} interpreted in the Stratonovich sense. Assuming that the coefficients $H_{0},\bh\in(0,1)^{d+1}$ verify
\begin{equation}\label{eq:strato-subcritical-regime}
d+\frac23 <2H_{0} + H \leq d+1,
%\quad\text{or otherwise stated}\quad
%d-H<2H_{0}-\frac23 ,
\end{equation}
a global Stratonovich solution $u$ can indeed be constructed in a suitable weighted space, using the regularity structure formalism and a renormalization procedure (see Section \ref{sec:existence-strato} for further details).} 
%\hre{In this setting it should be observed that a full global existence and uniqueness theory for equation~\eqref{eq:she-intro} is not directly available in the literature, although one might argue that such existence and uniqueness result might be reduced to the computations in \cite{CH,HL}. We have carried all the necessary computations leading to existence and uniqueness of the Stratonovich solution in the Arxiv version of this paper~\cite{CDOT1}, within the landmark of regularity structures. However, we have decided to skip the details of our constructions for the current paper, in order to keep the article to a reasonable size. The reader is also referred to the forthcoming paper~\cite{CDOT2} for further details.}

\smallskip

{
As for the Skorohod case mentioned above,
 we are mostly interested here in getting some information about the moments of the Stratonovich solution $u$. 
This problem is challenging in a renormalized context, and
to the best of our knowledge unaddressed in the literature. We are only aware of the reference~\cite{GH} about existence of moments in small time for the 2-dimensional spatial white noise case. Our next contribution, which is summarized below, aims at filling this gap. 
}

\

\begin{contribution}\label{cont:moments-strato}
{Under the assumption \eqref{eq:strato-subcritical-regime}}, let $u_{t}(x)$ be the (renormalized) Stratonovich solution of equation \eqref{eq:she-intro} constructed from the regularity structure formalism. Then the following holds true:
\begin{enumerate}[wide, labelwidth=!, labelindent=0pt, label=(\roman*)]
\setlength\itemsep{.1in}
\item[]\label{cont:moments-strato-subcritical}
Under the {additional} subcritical assumption \eqref{eq:sko-subcritical-regime}, the random variable $u_{t}(x)$ admits moments of all orders. More specifically, we have
\begin{equation*}
\be\big[|u_t(x)|^p\big]<\infty,\quad\mathrm{for\ all}\ t\geq0, \, x\in\R^d, \mathrm{and}\ p\geq1.
\end{equation*}

%\item
%In the critical case defined by \eqref{eq:sko-critical-regime} and for a fixed $p>1$, the random variable $u_{t}(x)$ admits a moment of order $p$ for $t$ small enough. \aurel{$\longrightarrow$ \textbf{I am not sure we handle this situation.}} 

\end{enumerate}
\end{contribution}

\begin{remark}
{To the best of our knowledge, the above Contribution \ref{cont:moments-strato} is the first result establishing moments of any order in a renormalized Stratonovich setting.}
\end{remark}

\begin{remark}
{
On top of~\cite{GH} let us mention the work~\cite{AC2015}, in which  the authors study equation~\eqref{eq:she-intro} driven by a spatial white noise on a $2$-$d$ torus. The exponential tail bounds for the minimal eigenvalue of the Anderson operator established in~\cite{AC2015} would certainly lead to a blowup of moments result for equation~\eqref{eq:she-intro}. However, their setting is restricted to the spatial white noise on a compact 2-d space, while we are considering a more general class of noises on the whole space $\R^d$ for all $d$. The 2-d white noise analyzed in~\cite{AC2015} happens to be a critical case within our more global picture. It should also be highlighted that our approach aims at obtaining the critical exponential integrability of the solution, instead of a simple non quantified exponential integrability.
}
\end{remark}

Gathering our Contributions~\ref{contribution1} and \ref{cont:moments-strato}, our main goal in this paper is thus to exhibit the complexity of the moments problem for the PAM in very rough environments. 

\smallskip

Let us say a few words about the methodology invoked in order to achieve the main contributions summarized above. We have relied on the following tools:
\begin{enumerate}[wide, labelwidth=!, labelindent=0pt, label=(\alph*)]
\setlength\itemsep{.05in}
\item 
In order to analyze the moments of the Skorohod solution in the critical case (i.e Contribution~\ref{contribution1}), we hinge on the Feynman-Kac representation of $u^{\di}$. This representation involves intersection local times of a Brownian motion weighted by the covariance function of the noise $\dw$. We then proceed to a thorough analysis of those intersection local times, thanks to a delicate truncation procedure involving a Girsanov type transform on the Brownian paths. Notice that the basic aim of the truncation mechanism is (in the end) to apply some classical asymptotic results for Feyman-Kac functionals. For the case of $p$-th moments for $p\in(1,2)$ (that is item~\ref{cont:critical-p-less-2} in our Contribution~\ref{contribution1}), we use an additional ingredient based on $\cs$-transforms from white noise analysis.

%\item
%As mentioned earlier, the existence of a renormalized solution to the Stratonovich version of~\eqref{eq:she-intro} is obtained within the regularity structures framework. In this context we have to adapt the general theory to our situation in $\R^{d}$. Then our main task is to effectively construct a $K$-rough path containing the Gaussian iterated integrals necessary to build proper modeled distributions related to the heat equation. This will be achieved in Appendix~\ref{sec:stochastic-constructions}. 

\item
The fact that the Stratonovich solution admits moments of any order (Contribution~\ref{cont:moments-strato} item~\ref{cont:moments-strato-subcritical}) stems again from Feynman Kac representations. More precisely, the difference between the Feynman-Kac representations of $u$ and $u^{\di}$ corresponds to some fluctuations of the Brownian intersection local times. Then the bulk of our analysis consists in quantifying the fact that the intersection local time fluctuations are negligible with respect to the term involving $\dw$.
\end{enumerate}

\noindent
As one can see, Feynman-Kac representations are the key ingredient in order to get our main contributions. Those are combined throughout the paper with regularity structures and Malliavin calculus elements. As a conjecture,  let us also recall from Remark \ref{rmk: contri 1} that $t_{0}(p)$ defined by~\eqref{intro-t0} seems to be the critical time for the moments of $u^{\di}$ when $p<2$ (as well as for the moments of the Stratonovich solution $u$ in the critical case). We hope to prove this claim in the future, although it currently seems to resist our method of analysis.

\smallskip

{The study is organized as follows. First, in Section \ref{sec:malliavin}, we will recall some basics about the Malliavin calculus setting and the Skorohod integration procedure. This preliminary material will in fact set the stage for Section \ref{section:Skorohod case}, where equation \eqref{eq:she-intro} will be investigated in the Skorohod sense, leading to the above-described Contribution \ref{contribution1}. {Then, in Section \ref{sec:strato}, we will turn to the Stratonovich interpretation of the model, starting with a brief reminder on how the equation is interpreted and solved in a rough regime (Section \ref{sec:existence-strato}). Our second main Contribution \ref{cont:moments-strato} will be the topic of the subsequent Section \ref{section:moments-estim}, which will conclude the study.} %Finally, Appendix \ref{sec:stochastic-constructions} will be devoted to the proof of the main technical result of Section \ref{sec:strato}, namely the construction of a so-called $(\al,K)$-rough path above the fractional noise $\dot{W}$.}

%\paragraph{\textbf{Warning.}}
%For sake of conciseness, we have sent many of our technical proofs to the Appendix. The referees are sent to this section if they wish to get more details about our computations. However, in our mind the Appendix would completely disappear from a potential published version in order to keep the article to a reasonable size.

\begin{notation}\label{not:general}
In the sequel we set $(\oom,\cf,\bp)$ for the probability space related to $W$, with $\be$ for the related expected value. We denote by $(\hat{\oom},\cg,\PP)$ the probability spaces corresponding to the Brownian motions in Feynman-Kac representations, with a related expected value $\E$. 
The heat kernel on $\R^{d}$ is denoted by $p_{t}(x)$, and recall that
\begin{equation}\label{eq:heat-kernel}
p_{t}(x) = \frac{1}{(2\pi t)^{d/2}} \, \exp\lp -\frac{|x|^{2}}{2t}  \rp.
\end{equation}
Also notice that the inner product of $a,b\in\R^d$ is written as $a\cdot b$ throughout the paper.
In the following we will often deal with product measures, for which we adopt the following convention: for $m, n\geq 1$, a measure $\nu$ on $\R^m$ and $\xi=(\xi_1,\dots,\xi_n)\in (\R^m)^n$,  we set
\begin{align}\label{convention product measure}\nu(d\xi)=\nu^{\otimes n}(d\xi)=\prod_{k=1}^n\nu(d\xi_k).\end{align}
Finally, recall that we write $\bh$ for the vector of Hurst parameters $(H_1,\dots,H_d)$.
\end{notation}

\section{Elements of Malliavin calculus}\label{sec:malliavin}

In this section we recall the basic Malliavin calculus notation which will be invoked in the forthcoming computations. The reader is {sent} to \cite{HHNT} and \cite{Nu-bk} for more details.

Let us start by introducing some basic notions on Fourier transforms of functions.  The space of real valued infinitely differentiable
functions with compact support is denoted by $\mathcal{D} ( \mathbb{R}^{d})$ or
$\mathcal{D}$. We write $\mathcal{S} (
\mathbb{R}^{d})$, or simply $\mathcal{S}$ for the space of Schwartz functions. Its dual, the space of tempered
distributions, is $\mathcal{S}' ( \mathbb{R}^{d})$ or $\mathcal{S}'$. If $f$ is
a vector of $n$ tempered distributions on $\mathbb{R}^{d}$,  then we write $f
\in \mathcal{S}' (\mathbb{R}^{d}, \mathbb{R}^n)$. In the sequel we will play with two different kinds of Fourier transforms on $\R^{d+1}$. Namely for a function $f(t,x)$ on $\R^{d+1}$, the Fourier transform  on the full space-time domain $\R^{d+1}$ is
defined with the normalization
\begin{align}\label{def:Fourier transform} 
\mathcal{F} f (\eta, \xi) = \int_{\mathbb{R}^{d+1}} e^{- \imath\, 
   (t\eta+\xi\cdot x) } f (t,x) dtd x, 
   \end{align}
so that the inverse Fourier transform is given by $\mathcal{F}^{- 1} u (\eta, \xi)
= \frac{\mathcal{F}f (-t, - \xi)}{(2\pi)^{d+1}}$. Our analysis will also rely on a Fourier transform denoted by  $\mathcal{F}^{\textsc s}$, which is defined on the spatial variables only.  It is given by
\begin{align}
\label{def:space Fourier transform} 
\mathcal{F}^{\textsc s} f (t, \xi) = \int_{\mathbb{R}^{d}} e^{- \imath\, 
   \xi\cdot x } f (t,x) d x. 
  \end{align}

\subsection{Wiener space related to $W$}\label{intro Wiener space}

On a complete probability space $(\Omega,\cf,\bp)$ we consider a Gaussian noise $W$ encoded by a centered Gaussian family $\{W(\vp) ; \, \vp\in {\mathcal{D}}(\R_{+}\times\R^{d})\}$. As mentioned in the introduction (see relation \eqref{eq:rep-gamma-fourier}), the covariance structure of $W$ is given by
\begin{equation}\label{eq:covariance-W-with-Fourier}
\be\lc W(\vp) \, W(\psi) \rc
= 
\int_{\R_{+}^{2}}\int_{\R^{d}} 
\mathcal{F}^{\textsc s}\vp(s_1,\xi){\bar{\mathcal{F}^{\textsc s}}}\psi(s_2,\xi)\, \gamma_0(s_1-s_2)\, ds_1ds_2\, \mu(d\xi),
\end{equation}
where $\gamma_0$ and $\mu$ are respectively defined by \eqref{eq:def-gamma} and \eqref{eq:def-mu}, and where $\mathcal{F}^{\textsc s}\vp$ designates the spatial Fourier transform of $\vp$.
One can thus consider $W$ as an isonormal Gaussian family $\{W(\vp) ; \, \vp\in \ch\}$ on a space $\ch$ obtained as the completion of Schwartz functions with respect to the inner product given by the right hand side of~\eqref{eq:covariance-W-with-Fourier}.

\smallskip

We refer to \cite{Nu-bk}
for a detailed account of the Malliavin calculus with respect to a
Gaussian process, while we focus here on the basic definitions allowing to state our results in the remainder of the article.
We will denote by $D$ the derivative operator in the sense of
Malliavin calculus. That is, if $F$ is a smooth and cylindrical
random variable of the form
\begin{equation*}
F=f(W(\phi_1),\dots,W(\phi_n))\,,
\end{equation*}
with $\phi_i \in \mathcal{H}$, $f \in C^{\infty}_p (\R_{+}\times\R^{d})$ (namely $f$ and all
its partial derivatives have polynomial growth), then $DF$ is the
$\mathcal{H}$-valued random variable defined by
\begin{equation*}
DF=\sum_{j=1}^n\frac{\partial f}{\partial
x_j}(W(\phi_1),\dots,W(\phi_n)) \, \phi_j\,.
\end{equation*}
The operator $D$ is closable from $L^2(\Omega)$ into $L^2(\Omega;
\mathcal{H})$  and we define the Sobolev space $\mathbb{D}^{1,2}$ as
the closure of the space of smooth and cylindrical random variables
under the norm
\[
\|DF\|_{1,2}=\left( \be[F^2]+\be[\|DF\|^2_{\mathcal{H}}] \right)^{1/2}\,.
\]

For any integer $n\ge 0$ we denote by $\mathcal{H}_n$ the $n$th Wiener chaos of $W$. We recall that $\mathcal{H}_0$ is simply  $\R$ and for $n\ge 1$, $\mathcal {H}_n$ is the closed linear subspace of $L^2(\Omega)$ generated by the random variables $\{ H_n(W(h));h \in \mathcal{H}, \|h\|_{\mathcal{H}}=1 \}$, where $H_n$ is the $n$th Hermite polynomial.
For any $n\ge 1$, we denote by $\mathcal{H}^{\otimes n}$ (resp. $\mathcal{H}^{\odot n}$) the $n$th tensor product (resp. the $n$th  symmetric tensor product) of $\mathcal{H}$. Then, the mapping $I_n(h^{\otimes n})= H_n(W(h))$ can be extended to a linear isometry between    $\mathcal{H}^{\odot n}$ (equipped with the modified norm $\sqrt{n!}\| \cdot\|_{\mathcal{H}^{\otimes n}}$) and $\mathcal{H}_n$.

Consider now a random variable $F\in L^2(\Omega)$ which is measurable with respect to the $\sigma$-field  $\mathcal{F}^W$ generated by $W$. This random variable can be expressed as
\begin{equation}\label{eq:chaos-dcp}
F= \be \lc F\rc + \sum_{n=1} ^\infty I_n(f_n),
\end{equation}
where the series converges in $L^2(\Omega)$. Moreover, the functions $f_n \in \mathcal{H}^{\odot n}$ are determined by $F$.  Identity \eqref{eq:chaos-dcp} is called the Wiener chaos expansion of $F$.

\subsection{Extended Skorohod integrals}

The standard definition of Skorohod's integral is obtained by a duality relation in $L^{2}(\oom)$. Namely, we denote by $\delta^{\di}$ the $L^{2}$-adjoint of the derivative operator given
by the duality formula
\begin{equation}\label{dual}
\be \lc \delta^{\di} (u) \, F \rc =\be \lc \langle DF,u
\rangle_{\mathcal{H}}\rc , 
\end{equation}
for any $F \in \mathbb{D}^{1,2}$ and any element $u \in L^2(\Omega;\mathcal{H})$ in the domain of $\delta^{\di}$. More specifically, the domain of $\delta^\di$ is defined as the set of processes {$u \in L^2(\Omega;\mathcal{H})$} such that for all cylindrical functions $F$ we have
$$|\be[\langle DF, u\rangle_{\mathcal{H}}]|\leq C_u\|F\|_{L^2(\Omega)}.$$
The operator $\delta^{\di}$ is
also called the {\it Skorohod integral} because in the 
Brownian motion case, it coincides with an extension of the It\^o
integral introduced by Skorohod.  %If $DF$ and $u$ are almost surely in $\ch$,
%then the duality formula~\eqref{dual} can be written  using the
%expression of the inner product in $\mathcal{H}$, that is:
%\begin{equation} \label{eq:duality-specific}
%\be \lc \delta^{\di}(u)F \rc =
%\be \lc \lla DF, u \rra_{\ch} \rc.
%\end{equation}

In the sequel we consider random variables which are only $L^{p}$-integrable with $p\in(1,2)$.
We thus need the following extension of the Skorohod integral with respect to the noise $W$.

\begin{definition}  \label{def1}
Fix $p>1$. Let $u=\{ u(t,y), t\ge 0, y\in \R^{d}\}$ be a random field such that 
\[
 \left\|  \|  u({t,y})\|_{\ch} \right\|_{L^p(\Omega)}   <\infty.
 \]
  Then we say that  $u$ is Skorohod integrable in $L^p$, denoted by $u\in {\rm Dom}_p\delta^{\di}$, if for any smooth and cylindrical random variable $F\in \mathcal{S}$, we have 
\[
\Big|\be \lc  \lla DF , \, u \rra_{\ch} \rc \Big|  \le c_u \| F\|_{L^q(\Omega)},
\]
where $\frac 1p + \frac 1q =1$.  We say that $u$ is Skorohod integrable if $u\in {\rm Dom} \delta^{\di} =\cup_{p>1} {\rm Dom}_p\delta^{\di}$.
If $u\in {\rm Dom}_p\delta^{\di}$, the Skorohod integral 
$$
\delta^{\di}(u)=\int_{\R_{+}\times\R^{d}} u_s(y) d^{\di}W_{s}(y)
$$ 
is the random variable in $L^p(\Omega)$ defined by the duality relationship
\[
 \be \lc  \lla DF , \, u \rra_{\ch} \rc  
 =   \be \lc \delta^{\di}(u) \, F \rc,
 \quad\text{for all}\quad
 F\in \mathcal{S}.
 \]
\end{definition}

\section{Skorohod Case}\label{section:Skorohod case}
This section focuses on equation \eqref{eq:she-intro} interpreted in the Skorohod sense. In Section \ref{subsection exist and uniq} below we first recall some existence and uniqueness results for the equation. Then we will give some information about the moments in the subcritical case.
\subsection{Existence and uniqueness results}\label{subsection exist and uniq}

In this section we first give some existence and uniqueness results for equation \eqref{eq:she-intro} considered in the Skorohod sense. Those results are mainly taken from \cite{Ch18}, although we will also deal with extended solutions in the sense of Definition \ref{def1}. Specifically, we consider equation \eqref{eq:she-intro} considered in the Skorohod sense, that is 
\begin{equation}\label{eq:she-sko}
\partial_{t} u _{t}(x) = \frac12 \Delta u_{t}(x) + u_{t}(x) \, \di \dot{W}_{t}(x), \quad t\in\R_{+}, x\in \R^{d}.
\end{equation}
%We shall recall some basic results about \eqref{eq:she-sko} in the critical regime \tbc{[Aurelien] (and also in the subcritical regime, isn't it?)}, then see when its $L^{p}$ moments become divergent for $p\ge 1$.
Recall from Notation \ref{not:general}  that we  denote by $p_{t}(x)$ the
$d$-dimensional heat kernel 
$p_{t}(x)=(2\pi t)^{-d/2}e^{-|x|^2/2t} $, for any $t > 0$, $x \in \R^d$.
We define the solution of equation~\eqref{eq:she-sko} as follows.

\begin{definition}\label{def3}
Let $T>0$.  Consider a random field $u=\{u_{t}(x); 0\leq t\leq T, x \in \mathbb{R}^{d}\}$ such
that there exists $p>1$ satisfying
\begin{equation}  \label{eq1}
\sup_{0\le s\le T}\sup_{y\in \R^d}  \| u_s(y)\|_{L^p(\Omega)} <\infty .
\end{equation}  
Then $u$ is said to be a  mild solution
to equation \eqref{eq:she-sko} up to time $T$ with initial condition   $u_0\in C_b(\mathbb{R}^{d})$, 
if for any $t\ge0$ and $x\in\mathbb{R}^{d}$,
 the process $\{  u_{s}(y); \, s\ge 0, \,  y \in \mathbb{R}^d\}$ is Skorohod
integrable  in the sense  of Definition~\ref{def1}, and the following equation holds for any $t\ge 0$ and $x\in \R^d$:
\begin{equation}\label{eq:mild-sko}
u_{t}(x)=  p_t u_0(x)+ \iot \int_{\mathbb{R}^d} p_{t-s}(x-y)u_{s}(y)   d^{\di}W_{s}(y).
\end{equation}
\end{definition}

If we impose $p=2$ in Definition  \ref{def3}, then our definition coincides with the usual notion of Skorohod solution for equation \eqref{eq:she-intro}. In the subcritical regime $L^{p}$ moments are available for all $p\ge 2$, and therefore equation \eqref{eq:she-sko} can be solved in the usual sense. We recall this result (taken from \cite{Ch18}) for {the} sake of completeness. %\hb{(Here mention something about the extension to $H_{0}<1/2$)}.

\begin{proposition}\label{prop:existence-sko-subcritical}
We assume that the coefficients $H_{0},\ldots,H_{d}$ satisfy the subcritical relation~\eqref{eq:sko-subcritical-regime} and that $H_{0}>\frac12$. Then equation  \eqref{eq:she-sko} admits a unique solution, considered in the sense of Definition~\ref{def3} with $p=2$.\end{proposition}

\begin{remark}
{Proposition \ref{prop:existence-sko-subcritical} has been extended in \cite{Ch19} to the case $H_0<1/2$.}
\end{remark}

In contrast with the subcritical situation, in the critical case one can only guarantee the existence of a usual Skorohod solution up to a critical time. This is the content of the following proposition, also borrowed from \cite{Ch18}.

\begin{proposition}\label{prop:local-existence-uniqueness-sko}
Assume that the critical condition \eqref{eq:sko-critical-regime} on $H_{0}$ and $\bh$ holds true, where we recall that the vector $\mathbf{H}$ is introduced in Notation \ref{not:general}.
Then one can find a strictly positive time $\ts=\ts(H_{0},\bh,d)$ such that there exists a unique $L^{2}(\oom)$-mild solution $u^{\di}$ of equation~\eqref{eq:mild-sko}  for $t\in[0,\ts)$.
\end{proposition}

Proposition \ref{prop:local-existence-uniqueness-sko} opens the way to a possible definition of the solution to equation \eqref{eq:mild-sko} in the $L^{p}$ sense for $p\in(1,2)$. While the existence of such a solution relies heavily on moment bounds which are obtained through Feynman-Kac representations, uniqueness holds true by just invoking the fact that we are dealing with a linear equation. This is summarized in the following proposition.

\begin{proposition}\label{thm: uniqueness-extended Skorohod}
As in Proposition \ref{prop:local-existence-uniqueness-sko}, we assume that the critical condition \eqref{eq:sko-critical-regime} on $H_{0}$ and $\bh$ holds true. For $p\in(1,2)$, suppose that one  can find a strictly positive time $\ts_{p}=\ts(p,H_{0},\bh,d)$ such that {there exists a $L^{p}(\oom)$-mild solution} $u^{\di}$ of equation~\eqref{eq:mild-sko}  on the interval $[0,\ts_{p})$. Then this solution is unique.
\end{proposition}

\begin{proof}
Suppose that $u$ and $v$ are two mild solutions in the $L^{p}$ sense (given by Definition \ref{def3}) with the same initial condition. Set $w=u-v$. Then, $w$ satisfies
\begin{equation}\label{f1}
w_{t}(x)=   \int_0^t  \int_{\mathbb{R}^d}   
 p_{t-s}(x-y) w_{s}(y)  \, d^{\di}W_{s}(y).
\end{equation}
Iterating this relation we get 
\begin{equation}\label{f2}
 w_{t}(x)=   \int_0^t  \int_{\mathbb{R}^d} 
  \left(  \int_0^s  \int_{\mathbb{R}^d}  
 p_{t-s}(x-y)   
 p_{s-r}(y-z)w_{r}(z)   d^{\di} W_{r}(z) \right)  d^{\di}W_{s}(y).
\end{equation}
 That is,  $w_{t}(x)$ is an iterated Skorohod integral.  By Proposition 2.6  and Proposition 2.7   in    \cite{Nu-Za}  this iterated Skorohod integral coincides with a double Skorohod integral (notice that in the reference \cite{Nu-Za} these results are proved for Skorohod integrals in $L^2(\Omega)$, but they can be easily extended to the Skorohod integral in $L^p(\Omega)$ introduced in Definition \ref{def1}). Moreover, using \eqref{f1}, \eqref{f2} and iterated versions in higher order chaoses, one can prove that $w_{t}(x)$ is orthogonal to any element in a finite Wiener chaos. This does not imply immediately that $w_t(x)=0$, because $w_{t}(x) \in L^p(\Omega)$ for some $p>1$. However, given any random variable $F\in L^q(\Omega)$ where $\frac 1p +\frac 1q=1$, we can approximate $F$ in  $L^q$ by a sequence $F_n$ of elements on a finite Wiener chaos. Then, $\be [w_t(x)F] =\lim_{n\rightarrow \infty} \be [w_t(x) F_n]=0$, which implies $w_t(x)=0$.
 \end{proof}

%\section{Moments estimates for the Skorohod equation}\label{sec:skoro-mom}

 %In this section we discuss the existence of moments for the Skorohod equation \eqref{eq:she-sko}. We first recall the Feynman-Kac representation of moments results in Section \ref{sec:fk-representation-sko}. Then we give some existence of moments results in the subcritical case as well as some bounds on the blowup time for $L^{p}$-moments in the critical case.

 \subsection{Feynman-Kac representation}\label{sec:fk-representation-sko}
 
In order to analyze the moments of $u^\di_t(x)$ and compare the Skorohod and Stratonovich settings for equation \eqref{eq:she-intro}, we first give a Feynman-Kac representation for an approximating sequence of the solution $u^{\di}$ to equation \eqref{eq:she-sko}. To this end, let $p_\epsilon$ be a $d+1$ dimensional heat kernel given by \eqref{eq:heat-kernel}. To separate its first coordinate from the others, we write $p_\epsilon(\hat{x})=p_\epsilon(x_0,{x})$ with $\hat{x}=(x_0,{x})\in\R^{d+1}$. Consider now the smoothed noise 
\begin{align}\label{smoothed noise}
W^\epsilon_t(x)=\lastchange{(p_\epsilon*\dot{W})_t(x)}=W(p_\epsilon(t-\cdot,x-\cdot)).
\end{align}
In the same spirit as \eqref{eq:def-gamma}, we denote by $\gamma_0^\epsilon$ and $\gamma^\epsilon$ the corresponding covariance functions (in time and space) of $W^\epsilon$, with $\mu_0^\epsilon$ and $\mu^\epsilon$ their spectral measures. Similar to \eqref{eq:def-mu}, it is readily checked that 
\begin{equation}\label{eq:def-mu-epsilon}
\mu_{0}^{\ep}(d\la)= c_{0} \, e^{-\lastchange{\epsilon}|\lambda|^2}|\la|^{1-2H_{0}}d\lambda,
\quad\text{and}\quad
\mu^{\ep}(d\xi) = c_{\bh} e^{-\lastchange{\epsilon}|\xi|^2}\prod_{j=1}^{d} |\xi_{j}|^{1-2H_{j}}d\xi,
\end{equation}
With the above notions in hand, we set
\begin{equation}\label{eq:def-Vt-epsilon-2}
V_{t}^{\ep,B}(x) = \iot  W^\epsilon_{t-s}(B^x_s)ds=W\left(\phi\right),
\end{equation}
where the second notation $W(\phi)$ corresponds to the Wiener integral introduced in \eqref{eq:covariance-W-with-Fourier} and where the function $\phi$ is given by
\begin{align}\label{eq:phi}
\phi(s,y)=\int_0^tp_{\epsilon}(u-s,B_{t-u}^x-y)du.
\end{align}
Moreover, in relation \eqref{eq:phi} $B^{x}$ stands for a $d$-dimensional Brownian motion independent of $W$ with initial condition $x\in\R^{d}$. The functional \eqref{eq:def-Vt-epsilon-2} gives rise to an approximate solution of the Skorohod equation, as stated in the following proposition:

\begin{proposition}\label{approx solution}
For $\ep>0$, $t\in\R_{+}$ and $x\in\R^{2}$, let $V_{t}^{\ep,B}(x)$ be defined by \eqref{eq:def-Vt-epsilon-2} and set 
\lastchange{\begin{equation}\label{eq:def-beta-epsilon}
\beta_{t}^{\ep,B} 
\equiv 
\lastchange{\int_{[0,t]^{2}} \int_{\R^{d+1}} e^{-\ep \, (|\xi|^{2}+|\lambda|^2)} e^{\ii  \lp \xi  \cdot  (B_{s_{1}} - B_{s_{2}})+\lambda(s_2-s_1)  \rp}
 \, \mu_0(d\lambda) \mu(d\xi) \, ds_{1} ds_{2}} ,
\end{equation}}
where we write $B^{0}=B$ for notational sake, where $\mu^\epsilon$ is introduced in \eqref{eq:def-mu-epsilon} and where $\ga_{0}^\epsilon$ is the inverse Fourier transform of $\mu_{0}^{\ep}$.
We define $u^{\ep,\di}$ on $\R_{+}\times\R^{d}$ by:
\begin{equation}\label{eq:approx-FK-sol-sko}
u _{t}^{\ep,\di}(x)
=
\E\!\lc e^{V_{t}^{\ep,B}(x) - \frac12 \beta_{t}^{\ep,B} } \rc,
\end{equation}
where we recall that the expectation $\E$ has been introduced in Notation \ref{not:general}.
Then the process $u^{\ep,\di}$ is the unique solution to the following Skorohod type equation:
\begin{equation}\label{eq:heat-sko-approx}
\partial_{t} u_{t}^{\ep,\di}(x) = \frac12 \Delta u_{t}^{\ep,\di}(x) 
+ u_{t}^{\ep,\di}(x) \di \dot W^{\ep}(x), \quad t\in\R_{+}, x\in \R^{d},
\end{equation}
where $\dot W^{\ep}$ is the smoothed noise $\dot W^{\ep}= \dot W * p_{\ep}$ defined in \eqref{smoothed noise}.
\end{proposition}

\begin{proof}

Similarly to the proof of Proposition 5.2 in \cite{HN}, we can show that 
\begin{equation*}
u _{t}^{\ep,\di}(x) = \E\lc
\exp\left(  V_{t}^{\ep,B}(x) - \frac12 \be\left[\lln V_{t}^{\ep,B}(x)\rrn^{2}\right]\right)
\rc.
\end{equation*}
Now a direct application of \eqref{eq:covariance-W-with-Fourier} reveals that
\begin{align}\label{a01}
\be\!\lc |V_{t}^{\ep,B}(x)|^{2} \rc \notag 
&=
\lastchange{\int_{[0,t]^{2}} \int_{\R^{d+1}} e^{-\ep \, (|\xi|^{2}+|\lambda|^2)} e^{\ii  \lp \xi  \cdot  (B_{s_{1}} - B_{s_{2}})+\lambda(s_2-s_1)  \rp}
 \, \mu_0(d\lambda) \mu(d\xi) \, ds_{1} ds_{2}} \notag \\
&=
\beta_{t}^{\ep,B},
\end{align}
which proves our claim.
\end{proof}

We now prove the convergence of the regularized Feynman-Kac representation to the solution of the Skorohod equation \eqref{eq:she-sko}.

\begin{proposition}\label{prop:cvgce-FK-Skorohod}
Let $u^{\ep,\di}$ be the process defined by \eqref{eq:approx-FK-sol-sko}, and recall that $u^{\di}$ designates the solution to equation \eqref{eq:she-sko} as given in Proposition \ref{prop:existence-sko-subcritical} \lastchange{or Proposition~\ref{prop:local-existence-uniqueness-sko}}. We  assume that one of the following situations is met:

\noindent
\emph{(i)}
The subcritical condition \eqref{eq:sko-subcritical-regime} is satisfied and $t$ is any positive number.

\noindent
\emph{(ii)}
The critical assumption \eqref{eq:sko-critical-regime} prevails and  $t\in[0,\ts)$, where $\ts$ is defined in Proposition~\ref{prop:local-existence-uniqueness-sko}.

\noindent
Then for all $x\in\R^{d}$, the random variable $u^{\ep,\di}_{t}(x)$ converges to $u^{\di}_{t}(x)$ in $L^{2}(\oom)$ as $\ep\to 0$.
\end{proposition}

\begin{proof}
We will prove the proposition under condition (ii), the other case being handled similarly. In order to show the $L^{2}$-convergence of $u_{t}^{\di,\ep}$,
we consider the quantity $\al_{t}$ (also denoted by $\al_{t}^{12}$ in the sequel) defined by
\lastchange{\begin{eqnarray}\label{eq:def-alpha-t}
\al_{t}
&\equiv& 
\int_{[0,t]^{2}} \int_{\R^{d}} e^{\ii  \lp \xi  \cdot  (B_{s_{1}}^{1} - B_{s_{2}}^{2})+\lambda(s_1-s_2)  \rp}\mu_0(d\la)\mu(d\xi) \, ds_{1} ds_{2} 
%& =&
 %\int_{[0,t]^{2}} \ga\lp  B_{s_{1}}^{1} - B_{s_{2}}^{2} \rp \ga_{0}(s_{1}-s_{2})  \, ds_{1} ds_{2},
\end{eqnarray}}
%where $\ga$ and $\ga_{0}$ are defined by \eqref{eq:def-gamma} and where the second relation is understood in the generalized sense.
{In the above, $B^1, B^2$ are two independent Brownian motions starting from $x$, which are also independent of the noise $W$.} Then in order to prove that $u^{\ep,\di}_{t}(x)$ converges in $L^{2}(\oom)$, it is sufficient to verify that
\begin{equation}\label{a10}
\lastchange{\E[e^{\alpha_t}]<\infty \quad \text{and}} \quad \lim_{\ep_{1},\ep_{2}\to 0} \be\left[
u _{t}^{\ep_{1},\di}(x) \, u _{t}^{\ep_{2},\di}(x)
\right]
= \lastchange{\E}\left[e^{\al_{t}} \right].
\end{equation}
In order to prove \eqref{a10}, we resort to expression \eqref{eq:approx-FK-sol-sko}, which yields
\begin{multline*}
\be\left[
u _{t}^{\ep_{1},\di}(x) \, u _{t}^{\ep_{2},\di}(x)
\right] \\
=
\be\!\lc 
\E\!\lc
\exp\lp V_{t}^{\ep_{1},B^{1}}(x) + V_{t}^{\ep_{2},B^{2}}(x) 
-\frac12\lp \beta_{t}^{\ep_{1},B^{1}} + \beta_{t}^{\ep_{2},B^{2}} \rp
\rp
\rc
\rc.
\end{multline*}
Applying Fubini's theorem we thus get:
\begin{multline}\label{a02}
\be\left[
u _{t}^{\ep_{1},\di}(x) \, u _{t}^{\ep_{2},\di}(x)
\right]  
=
\E\!\lc
\exp\lp \frac12\be\!\lc \lp V_{t}^{\ep_{1},B^{1}}(x) + V_{t}^{\ep_{2},B^{2}}(x) \rp^{2}  \rc
-\frac12\lp \beta_{t}^{\ep_{1},B^{1}} + \beta_{t}^{\ep_{2},B^{2}} \rp
\rp
\rc .
\end{multline}
Moreover, recall that $\beta_{t}^{\ep,B}=\be[ |V_{t}^{\ep,B}(x)|^{2} ]$ according to \eqref{a01}. Plugging this identity into~\eqref{a02}, we thus get
\begin{equation*} 
\be\left[
u _{t}^{\ep_{1},\di}(x) \, u _{t}^{\ep_{2},\di}(x)
\right]
=
\E\lc e^{\al_{t}^{\ep_{1},\ep_{2}}} \rc,
\end{equation*}
where the quantity $\al_{t}^{\ep_{1},\ep_{2}}$ is defined by
\begin{align}\label{a101}
&\al_{t}^{\ep_{1},\ep_{2}}
=
\be\!\lc V_{t}^{\ep_{1},B^{1}}\!(x) \, V_{t}^{\ep_{2},B^{2}}\!(x) \rc  \\
&=
\lastchange{\int_{[0,t]^{2}} \int_{\R^{d+1}} e^{-\frac{(\ep_{1}+\ep_{2})}{2} \, (|\xi|^{2}+|\lambda|^2)} e^{\ii  \lp \xi  \cdot  (B_{s_{1}}^{1} - B_{s_{2}}^{2})+\lambda(s_2-s_1)  \rp}
 \, \mu_0(d\lambda) \mu(d\xi) \, ds_{1} ds_{2}.} \notag
\end{align}
Notice that for fixed $\ep_{1},\ep_{2}>0$, the fact that $\al_{t}^{\ep_{1},\ep_{2}}$ is well-defined stems easily from the presence of the exponential term $e^{-(\ep_{1}+\ep_{2}) \, |\xi|^{2}}$ in the right hand side of \eqref{a101}.
In addition, \cite[inequality (3.1)]{Ch18} implies that $$\E[e^{\alpha_t}]<\infty,$$
for $t< t^*$. 
{Since 
$$\int_{[0,t]^2}\E\left[e^{\ii \lp \xi\cdot (B_{s_{1}}^{1} - B_{s_{2}}^{2})+\lambda(s_1-s_2)\rp }\right]ds_1ds_2
=
\left|\int_0^t\E\left[e^{\ii \lp \xi\cdot B_{s_1}^{1}+\lambda s_1\rp }\right]ds_1\right|^2\ge 0,
$$}an easy monotone convergence argument together with a Taylor series expansion of the exponential function yields
\begin{equation*}
\lim_{\ep_{1},\ep_{2}\to 0} \be\left[
u _{t}^{\ep_{1},\di}(x) \, u _{t}^{\ep_{2},\di}(x)
\right]
= 
\lim_{\ep_{1},\ep_{2}\to 0} \E\lc e^{\al_{t}^{\ep_{1},\ep_{2}}} \rc
=
\E\left[e^{\al_{t}} \right],
\end{equation*}
which is our claim \eqref{a10}. We have thus obtained
the $L^{2}$ convergence of $u^{\ep,\di}$.

We now claim that the process \lastchange{$u^\di_t(x)$} defined as the limit of $u^{\ep,\di}_t(x)$ is a mild solution to
 equation   \eqref{eq:she-sko} in the sense of Definition \ref{def3} with $p=2$.
First notice that  (\ref{eq1}) holds for the process $u^{\di}$ because of the relation
\[
\sup_{\ep >0} \sup_{t\le \ts}\sup_{x\in \R^2}  \| u^{\ep,\di}_t(x)\|_{L^2(\Omega)} <\infty,
\]
which stems from \eqref{a10}. Then owing to relation \eqref{dual}, $u^{\di}$ satisfies \eqref{eq:mild-sko} in the Skorohod sense if for any smooth and cylindrical random variable $F$, we have
\begin{equation}\label{a11}
\be \lc F \lp u^\di_t(x) - p_tu_0(x)\rp \rc
= \be \lc \lla DF, G \rra_{\ch} \rc  ,
\end{equation}
where the process $G$ is defined by
\begin{equation*}
G_{s}(y) = p_{t-s}(x-y) \lastchange{u^\di_s}(y) \1_{[0,t]}(s).
\end{equation*}
Relation \eqref{a11} is obtained by taking limits on a similar equation for $u^{\ep,\di}_t(x)$  (see~\cite[Theorem 3.6]{HHNT} for a similar argument).
\end{proof}

\subsection{Moments estimates in the subcritical case}
In this section we recall some moment estimates and representations established in \cite{Ch18} for the Skorohod equation. We label those results here since they will be invoked in order to get integrability results for the Stratonovich equation. We start with a bound on the $L^{p}$ moments.

\begin{proposition}\label{prop:finite-moments-subcritical}
Assume that the subcritical condition \eqref{eq:sko-subcritical-regime} is satisfied, and let $u^{\di}$ be the unique solution of equation \eqref{eq:she-sko}. Then for all $t\ge 0$ and $p\ge 1$ we have
\begin{equation*}
\be\lc |u^\di_{t}(x)|^{p}  \rc = c_{p,t} < \infty.
\end{equation*}
\end{proposition}

Let us now define a generalization of the quantity $\al_{t}$ defined by \eqref{eq:def-alpha-t}, which is used for moment representations. Namely for two coordinates $1\le j_{1},j_{2} \le \lastchange{p}$ we set
\begin{equation}\label{eq:def-alpha-j1j2}
\al_{t}^{j_{1}j_{2}}
=
\int_{[0,t]^{2}} \ga\lp  B_{s_{1}}^{j_{1}} - B_{s_{2}}^{j_{2}} \rp \ga_{0}(s_{1}-s_{2})  \, ds_{1} ds_{2},
\end{equation}
where $\gamma$ and $\gamma_0$ are the (generalized) covariance function given by \eqref{eq:def-gamma}, {and $B^j, j=1,\dots,p$ are independent Brownian motions starting from $x$}.
With this notation in hand and thanks to the representation \eqref{eq:approx-FK-sol-sko} for the approximation $u^{\ep,\di}$ of $u^{\di}$, we also get the following representation for the moments of $u_{t}^{\di}(x)$.

\begin{proposition}\label{prop:FK-representation}
We suppose that the subcritical condition \eqref{eq:sko-subcritical-regime} is fulfilled, and recall that $u^{\di}$ is the unique solution of equation \eqref{eq:she-sko}. Then for all $t\ge 0$ and any integer $p\ge 2$ we have
\begin{align}\label{eq: FK moments-sko}
\be\lc |u^\di_{t}(x)|^{p}  \rc = 
\E\lc \exp\lp \sum_{1\le j_{1}<j_{2}\le p }  \al_{t}^{j_{1}j_{2}} \rp  \rc,
\end{align}
where $\alpha^{j_1,j_2}_t$ is introduced in \eqref{eq:def-alpha-j1j2}.
{Moreover,  for any $\lambda>0$ one has
$$\E\lc \exp\lp \lambda  \al_{t}^{j_{1}j_{2}} \rp  \rc<\infty.$$
}
\end{proposition}
%%%%%%%%% note used if we do not try to compare asymptotics for Stratonovich %%%%

%Once the representation of moments in \eqref{eq: FK moments-sko} is given, the following asymptotic result is obtained through some intricate optimization arguments.
%\begin{proposition}\label{prop:lyapou-skorohod}
%We suppose that the assumptions of Proposition \ref{prop:FK-representation} prevail. Then for all $p\ge 2$ the following limit holds true,
%\begin{align}\label{eq:lyapou-skorohod}
%\lim_{t\to\infty} t^{-\frac{H+2H_{0}-1}{H}}
%\ln\lp\be\lc |u^\di_{t}(x)|^{p}  \rc\rp = c_{p,H,H_{0}}   ,
%\end{align}
%with an explicit expression for $c_{p,H,H_{0}}$.
%\end{proposition}

%%%%%%%%%%%%%%%%%%%%%%%%%%%%%%%%%%%%%%%%%%%%

\subsection{Critical moments for $\mathbf{p\ge 2}$}
We assume throughout this section that the critical condition \eqref{eq:sko-critical-regime} is satisfied. In this case we will see that the moments of $u^{\di}$ blow up after a critical time $\ts$ and we shall identify the time $\ts$ in some cases. Before we proceed to the proof of this fact, let us introduce some functional spaces and inequalities of interest.

\begin{definition}\label{def: A_d}
Let $W^{1,2}(\R^{d})$ be the usual Sobolev space of order $(1,2)$ in $\R^{d}$. We define a subset $\ca_{d}$ of functions defined on $\ou\times \R^{d}$ as follows:
\begin{multline*}
{\mathcal A}_d=\Big\{g:\ou\times \R^{d}\to\R;
\hskip.1in g(s,\cdot)\in W^{1,2}(\R^d),
\text{ and }
\int_{\R^d}g^2(s,x)dx=1 \text{ for all } s\in [0,1]\Big\}.
\end{multline*}
\end{definition}
It is readily checked that there exists a constant $C$ such that the following inequality holds  for all $g\in\ca_{d}$:
\begin{multline}\label{eq:funct-ineq} 
\int_{[0,1]^2}\int_{(\R^d)^2}\gamma_0(s-t)\gamma(x-y)g^2(s,x)g^2(t,y)dxdydsdt 
\le C^4
\int_{[0,1]\times\R^d}\vert\nabla_xg(s,x)
\vert^2dsdx.
\end{multline}
%Otherwise stated, if $\mathcal H$ designates the Cameron-Martin space defined by the inner product~\eqref{eq:covariance-W-with-Fourier}, we have {\color{red} (the following is not correct now. Should we delete it?)} \hb{Inequality \eqref{eq: sobolev ineq} is used after \eqref{eq: step 9 midstep}. Why is it wrong?}
%\begin{align}\label{eq: sobolev ineq}
%\|g^2\|^{1/2}_{\mathcal H}\leq C\|\nabla_xg\|^{1/2}_{L^2([0,1]\times\R^d)}.
%\end{align}
\begin{remark}\label{existence of kappa}
When the noise $W$ is time independent, that is when $H_0=1$ (hence $\gamma_0\equiv 1$), equality \eqref{eq:funct-ineq} can be established using the same method as in \cite{BCR09}. The general case then follows from \cite[Lemma 5.2]{Ch17}.
\end{remark}

\begin{notation}\label{not:kappa}
We call $\ka=\ka(H_{0},\bh)$ the best constant in inequality \eqref{eq:funct-ineq}.
\end{notation}

\begin{remark}\label{relation GN inequality}
The classical Gagliardo-Nirenberg inequality asserts that for $p(d-2)\leq d$ we have
\begin{equation}\label{eq:gagliardo}
\|f\|_{L^{2p}(\R^d)}\leq {\kappa}\|\nabla f\|_{L^2(\R^d)}^{\frac{d(p-1)}{2p}}\cdot\|f\|_{L^2(\R^d)}^{1-\frac{d(p-1)}{2p}}.
\end{equation}
The best constant $\kappa$ in the above inequality is called the Gagliardo-Nirenberg constant.  Now in equation \eqref{eq:funct-ineq}  consider the special case when $d=2, p=2,$ and $H_0=1, H_1=H_2=1/2$. In this situation it is clear that the quantity $\kappa$ defined in Notation \ref{not:kappa} coincides with the Gagliardo-Nirenberg constant in relation~\eqref{eq:gagliardo}. 
\end{remark}

We can now state the main result of this section, which identifies the exact time of blowup for the integer moments of $u_{t}(x)$ in the critical case.

\begin{theorem}\label{thm:explosion-critical}
We suppose that condition \eqref{eq:sko-critical-regime} is met for our indices $H_{0}$ and $H$. For any integer $p\ge 2$ we define a critical time  $t_{0}(p)$ by
$$
t_0(p)=
\lp \frac{1}{\ka^{4}(p-1)}  \rp^{{1}/{(2H_{0}-1)}},
$$
where we recall that $\ka$ is the constant introduced in Notation \ref{not:kappa}.  For $\epsilon>0$, let $u^{\epsilon,\diamond}$ be the unique solution to equation \eqref{eq:heat-sko-approx}. Then the following assertions holds true:
\begin{enumerate}[wide, labelwidth=!, labelindent=0pt, label=\textnormal{(\roman*)}]
\setlength\itemsep{.1in}

\item\label{critical-i} 
If $t<t_0\equiv t_0(2)=\kappa^{-4/(2H_0-1)}$, then for \lastchange{every $x\in\R^d$} the sequence $\{u^{\epsilon,\di}_t(x),\epsilon>0\}$ converges in $L^2(\Omega)$ to an element \lastchange{$u^\di_t(x)$}. The process $\{u_t^{\di}(x): \lastchange{t\in [0,t_0)}, x\in\R^d\}$ is said to be the solution of equation \eqref{eq:she-sko}. 
\item\label{critical-ii}  
For any $p\geq2$, if $t<t_0(p)$ the sequence $\{u^{\epsilon,\di}_t(x);\epsilon>0\}$ also converges in $L^p(\Omega)$ to \lastchange{$u^\di_t(x)$}.

\item\label{critical-iii}  
For $p\in(1,\infty)$, if $t>t_0(p)$ and $x\in\R^d$ we have
$$\lim_{\epsilon\to0}\|u^{\epsilon,\di}_t(x)\|_{L^p(\Omega)}=\infty.$$

\end{enumerate}

%\begin{equation}\label{eq:explosion-Lp-sko}
%\begin{aligned}
%\be \lc |u_t^\di(x)|^p \rc
%\left\{\begin{array}{ll}\displaystyle<\infty,\hskip.3in \text{for }\ t<t_0(p),\\\\
%=\infty,\hskip.3in \text{for }\ t>t_0(p).\end{array}\right.
%\end{aligned}
%\end{equation}
\end{theorem}

\begin{remark}\label{rmk:critical-sko}
In the critical case when $p\in[1,2)$ the notion of solution to equation \eqref{eq:she-sko} or~\eqref{eq:mild-sko} is not as clear as in Section 
\ref{sec:fk-representation-sko}. However, let us mention the following facts:
\begin{enumerate}[wide, labelwidth=!, labelindent=0pt, label=(\alph*)]
\setlength\itemsep{.1in}

\item\label{it:sko-eq-limit} 
If $p\geq2$ and $t<t_0(p)$, then we can take limits in equation \eqref{eq:heat-sko-approx} and show that the limit $\{\lastchange{u^\di_t(x)}; t\geq 0, x\in\R^d\}$ still solves the Skorohod equation \eqref{eq:mild-sko}. We have not provided details for the sake of conciseness. 

\item Whenever $p\in(1,\infty)$ and $t>t_0(p)$ one can not take limits in equation \eqref{eq:mild-sko}, even if one resorts to the extended Skorohod setting of Definition \ref{def1}. However, item (iii) in Theorem~\ref{thm:explosion-critical} asserts that if we could give a meaning to equation \eqref{eq:mild-sko} by a regularization procedure, then its solution $u_t(x)$ would not belong to $L^p(\Omega)$.

\item 
Our current techniques do not allow to assert the $L^p$-convergence of $u^\epsilon$
when $p\in (1,2)$ and $t<t_0(p)$. Neither are we able to prove the existence of a $L^p$-solution of~\eqref{eq:mild-sko} for $p<2$, even if we invoke the extended Skorohod setting. On the other hand, the definition of $t_0(p)$ in Theorem \ref{thm:explosion-critical} clearly allows $p$ to be less than $2$. It is therefore reasonable to conjecture that when $p\in (1,2)$ and $t<t_0(p)$, $u^\epsilon$ converges in $L^p$ to the unique solution of equation \eqref{eq:mild-sko} in the extended Skorohod setting. Provided this can be achieved, Proposition \ref{thm: uniqueness-extended Skorohod} establishes the uniqueness of the aforementioned extended solution.

\end{enumerate}
\end{remark}

The proof of Theorem \ref{thm:explosion-critical} is split in the sections below.

\subsection{Proof of the $\mathbf{L^{p}}$-convergence in Theorem \ref{thm:explosion-critical}}\label{sec:proof-upper-bound}
In this section we work under the critical assumption \eqref{eq:sko-critical-regime}. Before we proceed to the proof of the relation $\be \lc |u^\di_t(x)|^p \rc<\infty$ for $t$ small enough, we first state a general sub-additivity result which is useful for our next computations. It is borrowed from \cite[Theorem 6.1]{Ch18}.
\begin{lemma}\label{sub}
Let $\nu(d\xi)$ be a measure on $\R^d$, and consider two $\R^d$-valued independent Brownian motions $B$ and $\widetilde{B}$. For $t\geq0$ we introduce the random variable ${\mathcal H}_t$ defined by
$$
{\mathcal H}_t=\int_{\R^d}\nu(d\xi)\bigg[\int_0^te^{i\xi\cdot B_s}ds\bigg]
\bigg[\int_0^te^{-i\xi\cdot \widetilde{B}_s}ds\bigg].
$$
Then for any $t_1,t_2>0$ and $\theta\in\R_+$, the following inequality holds true:
\begin{multline}\label{sub-2}
\sum_{n=0}^\infty \frac{\theta^n}{n!}\left\{\frac{1}{(t_1+t_2)^n}\E\left[{\mathcal H}_{t_1+t_2}^n\right]\right\}
\\
\le\bigg(\sum_{n=0}^\infty\frac{\theta^n}{n!}
\left\{\frac{1}{t_1^n}\E\left[{\mathcal H}_{t_1}^n\right]\right\}
\bigg)
\bigg(\sum_{n=0}^\infty \frac{\theta^n}{n!}\left\{\frac{1}{t_2^n}\E\left[{\mathcal H}_{t_2}^n\right]\right\}
\bigg)
\end{multline}
whenever the right hand side is finite. 
\end{lemma}

\begin{remark}
We could obviously have stated \eqref{sub-2} in the following exponential form,
\begin{align*}
\E\bigg[\exp\Big(\frac{\theta{\mathcal H}_{t_1+t_2}}{t_1+t_2}\Big)\bigg]\leq \E\bigg[\exp\Big(\frac{\theta{\mathcal H}_{t_1}}{t_1}\Big)\bigg]\,\E\bigg[\exp\Big(\frac{\theta{\mathcal H}_{t_2}}{t_2}\Big)\bigg].
\end{align*}
However, equation \eqref{sub-2} is the one which will be used for our computations below.

\end{remark}

We will now separate the upper bound computations for $u^\di_t(x)$ into a $L^2$ bound and $L^p$ bounds for $p>2$.
\subsubsection{Convergence for $p=2$}\label{sec:convergence-p-equal-2}
Our aim in this section is to prove item \ref{critical-i} in Theorem~\ref{thm:explosion-critical}. According to the Feynman-Kac representation \eqref{eq: FK moments-sko}, an upper bound on the $L^{2}$ moments of \lastchange{$u^\di_{t}(x)$} amounts to prove the following relation
\begin{align}\label{eq: moment bound alpha^12}
\E\left[\exp(\alpha^{12}_t)\right]<\infty,\quad
\text{for all}\ \ t<\frac{1}{\kappa^{{4/ (2H_0-1)}}}\equiv t_0(2),
\end{align}
where we recall that the constant $\kappa$ is defined in Notation \ref{not:kappa}. Also notice that the random variable $\alpha^{12}_t$ introduced in \eqref{eq:def-alpha-j1j2} will be written as
\begin{align}\label{eq: representing alpha12}
\alpha^{12}_t=\int_{[0,t]^2}\gamma_0(s-r) \, \gamma(B_{s}-\widetilde{B}_{r})\,
ds dr,
\end{align}
for two independent Brownian motions $B$ and $\widetilde{B}$.  In addition, notice that only the large positive values of $\alpha^{12}_t$ might be responsible \lastchange{for} the blowup of exponential moments. Therefore relation \eqref{eq: moment bound alpha^12} can be easily deduced from the following asymptotic tail behavior
%Let $0<t<\kappa(H_0)^{-{4/ 2H_0-1}}$ be fixed. It suffices to show that
\begin{align}\label{eq: tail estimate of alpha12}
\limsup_{b\to\infty}\frac{1}{b}\log \PP\left\{
\alpha_t^{12}\geq b\right\}\le -\frac{1}{\kappa^{4}\,t^{2H_0-1}}.
\end{align}
We now proceed to prove \eqref{eq: tail estimate of alpha12}, and we divide our proof in several steps.

\bigskip
\noindent{\it Step 1:  Scaling arguments.} We recall once again that $\gamma_0$ and $\gamma$ are introduced in \eqref{eq:def-gamma}. Formally they are given by
\begin{align}\label{recall gamma}
\gamma_0(u)=C_{H_0}\vert u\vert^{-\alpha_0},\hskip.1in\hbox{and}
\hskip.1in\gamma(x)
=C_{\bf H}\prod_{j=1}^d \vert x_j\vert^{-\alpha_j},\quad\text{where}\ \ \alpha_j=2-2H_j.
\end{align}
With this formal expression in mind, we consider $b>0$ and set $\sigma=\frac{b}{t}s, \rho=\frac{b}{t}r$ in expression~\eqref{eq: representing alpha12}.  Invoking the usual Brownian scaling and the fact that $\sum_{j=1}^d(H_j-1)=-1$ under our critical assumption \eqref{eq:sko-critical-regime}, we let the patient reader check that 
\begin{align}\label{eq: alpha^12 after rescaling}
\alpha^{12}_t\stackrel{(d)}{=}b^{-1}t^{2H_0-1}\!\int_{[0,b]^2}
\gamma_0\big(b^{-1}(\sigma-\rho)\big)\gamma\big(B_\sigma-\widetilde{B}_\rho\big)
d\sigma d\rho.
\end{align}
As mentioned in the introduction, expressions like \eqref{eq: alpha^12 after rescaling} are better expressed in Fourier modes, thanks to \eqref{eq:rep-gamma-fourier} and \eqref{eq:def-mu}. We get 
\begin{multline*}
\alpha_t^{12} 
\stackrel{(d)}{=}b^{-1}t^{2H_0-1} 
\times
\int_{\R^{d+1}}\mu_0(d\lambda)\mu(d\xi)
\bigg[\int_0^b e^{i(\lambda b^{-1}s+\xi\cdot B_s)}ds\bigg]
\bigg[\int_0^b e^{-i(\lambda b^{-1}s+\xi\cdot \widetilde{B}_s)}ds\bigg].
\end{multline*}
In order to write the above expression in a more compact way, let us define the following random variables,
\begin{align}\label{def: eta_b}
\eta_b(\lambda,\xi)=\int_0^b e^{i(\lambda b^{-1}s+\xi\cdot B_s)}ds,
\hskip.1in\hbox{and}\hskip.1in
\tilde{\eta}_b(\lambda,\xi)=
\int_0^b e^{-i(\lambda b^{-1}s+\xi\cdot \widetilde{B}_s)}ds, 
\end{align}
and denote by $Z_b$ the weighted integral of $\eta_b$ and $\tilde{\eta}_b$, namely
\begin{align*}
Z_b=\int_{\R^{d+1}}\eta_b(\lambda,\xi)\tilde{\eta}_b(\lambda,\xi)\mu_0(d\lambda)\mu(d\xi).
\end{align*}
We end up with the following identity in law
\begin{align}\label{eq: identity in law alpha and Z}
\alpha_t^{12}\stackrel{(d)}{=}b^{-1}t^{2H_0-1}Z_b.
\end{align}

We now go back to our main objective. Plugging \eqref{eq: identity in law alpha and Z} into \eqref{eq: tail estimate of alpha12}, a few elementary algebraic manipulations yield that \eqref{eq: moment bound alpha^12} can be reduced to the following bound,
\begin{align}\label{tail estimate Z_b midstep}
\limsup_{b\to\infty}\frac{1}{b}\log\PP\left\{|Z_b|^{1/2}\geq \frac{b}{t^{H_0-1/2}}\right\}
\le -\frac{1}{\kappa^4t^{2H_0-1}}.
\end{align}
Moreover, since $t$ is an arbitrary positive number above, we simply set $u=t^{-(H_0-1/2)}$. We get that \eqref{eq: moment bound alpha^12} is implied by the inequality
\begin{align}\label{c-5}
\limsup_{b\to\infty}\frac{1}{b}\log\PP\left\{\left\vert
Z_b\right\vert^{1/2}\ge ub\right\}\PP
\le -\frac{u^2}{\kappa^{4}},
\end{align}
which should hold for all $u>0$.

As a last preliminary step, we set up a cutoff procedure on the random variable $Z_b$. Namely for $\delta>0$ we define
\begin{align}\label{eq: cutoff mu_0}
\mu_0^\delta(d\lambda)= e^{-\delta\vert\lambda\vert^2}\mu_0(d\lambda),
\hskip.1in\hbox{and}\hskip.1in
\bar{\mu}_0^\delta(d\lambda)= (1-e^{-\delta\vert\lambda\vert^2})
\mu_0(d\lambda).
\end{align}
Accordingly, we also set
\begin{eqnarray}\label{ldp-5'}
Z_b^\delta&=&
\int_{\R^{d+1}}\eta_b(\lambda, \xi)
\tilde{\eta}_b(\lambda, \xi)\mu_0^\delta(d\lambda)\mu(d\xi) \notag \\
\bar{Z}_b^\delta&=&
\int_{\R^{d+1}}\eta_b(\lambda, \xi)
\tilde{\eta}_b(\lambda, \xi)\bar{\mu}^\delta_0(d\lambda)\mu(d\xi),
\end{eqnarray}
and notice that $Z_b=Z_b^\delta+\bar{Z}_b^\delta$.
We will treat $Z_b^\delta$ and $\bar{Z}_b^\delta$ separately in \eqref{c-5}.

\noindent{\it Step 2: Identification of a negligible term.} In this step we prove that the contribution of $\bar{Z}_b^\delta$ in \eqref{c-5} is negligible. Specifically, recalling that $\bar{Z}_b^\delta$ is defined by \eqref{ldp-5'}, given any $\epsilon>0$ and $L>0$ we prove that
\begin{align}\label{s-1}
\limsup_{b\to\infty}\frac{1}{b}
\log\PP\left\{|\bar{Z}_b^\delta|^{1/2}\ge \epsilon b\right\}\le -L,
\end{align}
when $\delta$ is small enough.   As a first step in this direction, introduce an additional parameter $\hat{\alpha}_0\in(\alpha_0,1)$, where we recall that $\alpha_0$ is defined by \eqref{recall gamma}.  According to the definition \eqref{eq: cutoff mu_0} of $\bar{\mu}_0^\delta$ we have
\begin{align*}
\bar{\mu}_0^\delta(d\lambda)=\left(1-e^{-\delta|\lambda|^2}\right)\mu_0(d\lambda)\le 
\big( 1-e^{-\delta\vert\lambda\vert^2}\big)^{\frac{\hat{\alpha}_0-\alpha_0}{2}}
\mu_0(d\lambda).
%&\le \delta^{{\bar{\alpha}_0-\alpha_0\over 2}}
% \vert\lambda\vert^{\bar{\alpha}_0-\alpha_0}
%\mu_0(d\lambda)= \delta^{{\bar{\alpha}_0-\alpha_0\over 2}}
%\hat{\mu}_0(d\lambda).%\hskip.2in \hbox{(say)}
\end{align*}
Owing to the elementary relation $1-e^{-x}\leq x$ for $x\geq 0$, recalling from \eqref{eq:def-mu} that $\mu_0(d\lambda)=c_{0}|\lambda|^{-(1-\alpha_0)}d\lambda,$ and setting $$\hat{\mu}_0(d\lambda)=c_{0}|\lambda|^{-(1-\hat{\alpha}_0)}d\lambda,$$ we get
\begin{align}\label{eq: bound bar mu_0^delta by hat mu_0^delta}
\bar{\mu}_0^\delta(d\lambda)&\le \delta^{\frac{\hat{\alpha}_0-\alpha_0}{2}}
 \vert\lambda\vert^{\hat{\alpha}_0-\alpha_0}
\mu_0(d\lambda)= \delta^{\frac{\hat{\alpha}_0-\alpha_0}{2}}
\hat{\mu}_0(d\lambda).%\hskip.2in \hbox{(say)}
\end{align}
We now compute the $n$-th moment of the random variable $\bar{Z}_b^\delta$. Starting from definition \eqref{ldp-5'} and invoking our convention \eqref{convention product measure} on product measures, for all $n\geq 1$ we have
\begin{align}\label{eq: nth moment of bar Z_b^deta}
\E\left[\left(\bar{Z}_b^\delta\right)^n\right]=\int_{(\R^{d+1})^n}\E\left[\prod_{k=1}^n\eta_b(\lambda_k, \xi_k)\tilde{\eta}_b(\lambda_k, \xi_k)\right]\bar{\mu}_0^\delta(d\lambda)\mu(d\xi)
.\end{align}
In addition the families $\{ \eta_b(\lambda_k,\xi_k); k\leq n\}$  and $\{ \tilde{\eta}_b(\lambda_k,\xi_k); k\leq n\}$ above are i.i.d, due to definition \eqref{def: eta_b} and the fact that $B, \widetilde{B}$ are two independent Brownian motion. Hence relation~\eqref{eq: nth moment of bar Z_b^deta} can be written as
\begin{align}\label{eq: nth moment of bar Z_b^deta 1}
\E\left[\left(\bar{Z}_b^\delta\right)^n\right]=\int_{(\R^{d+1})^n}\Big\vert \E\Big[\prod_{k=1}^n\eta_b(\lambda_k, \xi_k)\Big]\Big\vert^2\bar{\mu}_0^\delta(d\lambda)\mu(d\xi)
.\end{align}
Note that from the \lastchange{right-hand side} of the above identity, the $n$-th moment of $\bar{Z}_b^\delta$ is non-negative. Plugging inequality \eqref{eq: bound bar mu_0^delta by hat mu_0^delta} in the above identity, we have
\begin{align}\label{eq: nth moment of bar Z_b^deta 2}
\E\left[\left(\bar{Z}_b^\delta\right)^n\right]
\le \delta^{\frac{\hat{\alpha}_0-\alpha_0}{2}n}
\int_{(\R^{d+1})^n}\Big\vert\E\Big[\prod_{k=1}^n\eta_b(\lambda_k, \xi_k)\Big]\Big\vert^2\hat{\mu}_0(d\lambda)\mu(d\xi).
\end{align}

In order to bound the \lastchange{right-hand side} of \eqref{eq: nth moment of bar Z_b^deta 2}, we first go back to the definition \eqref{def: eta_b} of $\eta_b$ and set $\lambda:=b^{-1}\lambda$ therein. We let the patient reader check that the scaling can be read in~\eqref{eq: nth moment of bar Z_b^deta 2} as
\begin{multline}\label{eq: nth moment of bar Z_b^deta 3}
\E\left[\left(\bar{Z}_b^\delta\right)^n\right] 
\le \delta^{\frac{\hat{\alpha}_0-\alpha_0}{2}n} \,b^{n\hat{\alpha}_0}
\int_{(\R^{d+1})^n}\Big\vert\E\Big[\prod_{k=1}^n\int_0^b e^{i(\lambda_k s_k+\xi_k\cdot \widetilde{B}_{s_k})}ds_k\Big]\Big\vert^2\hat{\mu}_0(d\lambda)\mu(d\xi).
\end{multline}
Next we notice that the \lastchange{right-hand side} of \eqref{eq: nth moment of bar Z_b^deta 3} is upper bounded in \cite[relation (3.1)]{Ch18}. Indeed, with the correspondence $\hat{H}_0\equiv2^{-1}(2-\hat{\alpha}_0)$ and with relation \eqref{eq:sko-critical-regime} in mind,  we choose $\hat{\alpha}_0>\alpha_0$ such that  $4(1-\hat{H}_0)+(d_*-2H_*)<2$. Then a direct application of  \cite[relation (3.1)]{Ch18} yields %and recalling again the assumption $\sum_{j=1}^dH_j=d-1$ in our critical context, we get
%Let $\bar{\alpha}_0$ be chosen by the constrain that 
%$\bar{H}_0\equiv2^{-1}(2-\bar{\alpha}_0)$ satisfies 
%$$
%4(1-\bar{H}_0)+(d_*-2H_*)<2.
%$$
%Recall the moment bound (Theorem 1.2, \cite{Ch18}, with $H_0$ being replaced
%by $\bar{H}_0$)
%$$
%\int_{(\R^{d+1})^n}\hat{\mu}_0(d\lambda)\mu(d\xi)
%\bigg\vert\int_{[0,b]^n}\bigg(
%\E_0\prod_{k=1}^ne^{i\lambda_k s_k+\xi_k\cdot B(s_k)}\bigg)ds\bigg\vert^2
%\le C^nn!b^{(2\bar{H}_0-1)n}.
%$$
%Summarizing our computation, we have
\begin{align}\label{eq: nth moment of bar Z_b^deta 4}
&\E\left[\left(\bar{Z}_b^\delta\right)^n\right]
\le C^n \delta^{\frac{(\hat{\alpha}_0-\alpha_0)n}{2}}\,(n!)\,b^{n\hat{\alpha}_0}\,b^{n(1-\hat{\alpha}_0)}=C^n(n!)\,\delta^{\frac{(\hat{\alpha}_0-\alpha_0)n}{2}}b^n.
\end{align}
Starting from this inequality we can easily get a similar bound for $\E[|\bar{Z}_b^\delta|^n]$ by changing the constant $C$ in the \lastchange{right-hand side} of \eqref{eq: nth moment of bar Z_b^deta 4}.  Namely,  when $n$ is even, note that $\E[(\bar{Z}_b^\delta)^n]=\E[|\bar{Z}_b^\delta|^n]$ due to the fact that our measures $\mu$ and $\mu_0$ are symmetric. For any odd number $n=2k+1$, we just invoke the Cauchy-Schwartz
inequality, which yields
$$
\E\left[\vert\bar{Z}_b^\delta\vert^{2k+1}\right]\le \Big(\E\left[\left(\bar{Z}_b^\delta\right)^{2k}\right]\Big)^{1/2}
\Big(\E\left[\left(\bar{Z}_b^\delta\right)^{2k+2}\right]\Big)^{1/2}.
$$
We let the reader check that this slight elaboration yields \eqref{eq: nth moment of bar Z_b^deta 4} with $\E[|\bar{Z}_b^\delta|^n]$ on the \lastchange{left-hand side}.
Therefore, inserting \eqref{eq: nth moment of bar Z_b^deta 4} into a Taylor expansion for $x\mapsto e^x$, we get that, \lastchange{for some constant $C>0$,}
\begin{align}\label{integrability bar Z^delta_b}
Q\equiv
\sup_{b\ge 1}
\E\lc\exp\lp\frac{|\bar{Z}_b^\delta|}{C\delta^{\frac{\hat{\alpha}_0-\alpha_0}{2}}b}
\rp \rc
<\infty.
\end{align}

We can now go back to our claim \eqref{s-1}. Indeed, plugging \eqref{integrability bar Z^delta_b} into a standard application of Chebyshev's inequality, we obtain
\begin{align*}
\PP\left\{|\bar{Z}_b^\delta|^{1/2}\geq \epsilon b\right\}=\PP\left\{|\bar{Z}_b^\delta|\geq \epsilon^2b^2\right\}\leq Q\exp\left(-\frac{\epsilon^2b}{C\delta^{(\hat{\alpha}_0-\alpha_0)/2}}\right),
\end{align*}
from which \eqref{s-1}  is easily deduced by picking a small enough $\delta$. Taking the decomposition $Z_b=Z_b^\delta+\bar{Z}_b^\delta$ given by \eqref{ldp-5'}, our objective \eqref{c-5} is thus reduced to show the following bound for all $u, \delta>0$ (see e.g. \cite[Lemma 1.2.15]{DZ} for a general result yielding a proper identification of exponentially negligible terms),
\begin{align}\label{c-6}
\limsup_{b\to\infty}\frac{1}{b}\log\PP\big\{|Z_b^\delta|^{1/2}\ge ub\big\}
\le -\frac{u^2}{\kappa^{4}}.
\end{align}

\noindent{\it Step 3: Cutoff procedure in space.} In order to prove \eqref{c-6}, we further decompose $Z_b^\delta$ as follows.
\begin{align}\label{decomp: space cutoff}Z_b^\delta=Z^{\delta,M}_b+\bar{Z}^{\delta,M}_b,\end{align}
where $M\geq 1$ is an additional parameter and $Z_b^{\delta,M}, \bar{Z}_b^{\delta,M}$ are respectively defined by
\begin{align}\label{Z^delta,M}
Z_b^{\delta, M}=\int_{\R\times[-M, M]^d}\eta_b(\lambda,\xi)
\tilde{\eta}_b(\lambda,\xi)\mu^\delta_0(d\lambda)\mu(d\xi),
\end{align}and
\begin{align}\label{bar Z^delta,M}
\bar{Z}_b^{\delta, M}=\int_{\R\times([-M, M]^d)^c}\eta_b(\lambda,\xi)
\tilde{\eta}_b(\lambda,\xi)\mu^\delta_0(d\lambda)\mu(d\xi).
\end{align}
Similarly to the previous step, we will now prove that there exists $l>0$ such that
\begin{align}\label{space cutoff}
\limsup_{b\to\infty}\frac{1}{b}
\log\PP\bigg\{|\bar{Z}_b^{\delta,M}|^{1/2}\ge \epsilon b\bigg\}
\le -l.
\end{align}
To this aim we will first upper bound the moments of $\bar{Z}_b^{\delta,M}$ as in \eqref{eq: nth moment of bar Z_b^deta 4}. Namely, along the same lines as for \eqref{eq: nth moment of bar Z_b^deta 1} we have
\begin{align}\label{eq: space cutoff-nth moment }
\E\left[\left(\bar{Z}_b^{\delta,M}\right)^n\right]=\int_{(\R\times([-M,M]^d)^c)^n}\Big\vert \E\Big[\prod_{k=1}^n\eta_b(\lambda_k, \xi_k)\Big]\Big\vert^2{\mu}_0^\delta(d\lambda)\mu(d\xi),
\end{align}
which shows in particular that $\E[(\bar{Z}_b^{\delta,M})^n]\geq 0$ for all $n\geq 1$. Furthermore, recall from \eqref{def: eta_b} that 
\begin{align*}
\E\Big[\prod_{k=1}^n\eta_b(\lambda_k,\xi_k)\Big]&=\E\Big[\int_{[0,b]^n}\prod_{k=1}^n
e^{-(i\lambda_k b^{-1} s_k+\xi_k\cdot B_{s_k})}ds\Big]
=\int_{[0,b]^n} e^{-i\lambda_k b^{-1} s_k}
\E\Big[\prod_{k=1}^n e^{-i\xi_k\cdot B_{s_k}}\Big]ds.
\end{align*}
One can then trivially bound the terms $|e^{-i\lambda_ks_k}|$ by $1$ in the \lastchange{right-hand side} above and resort to the fact that $\E\left[\prod_{k=1}^ne^{-i\xi_kB_{s_k}}\right]\geq 0.$ This yields
\begin{align*}
\Big|\E\Big[\prod_{k=1}^n\eta_b(\lambda_k,\xi_k)\Big]\Big|\leq \int_{[0,b]^n}\E\Big[\prod_{k=1}^ne^{-i\xi_k\cdot B_{s_k}}\Big]ds.
\end{align*}
Reporting this inequality into \eqref{eq: space cutoff-nth moment } we get
\begin{align}\label{eq: bound bar Z^delta M by H}
&\E\big[(\bar{Z}_b^{\delta,M})^n\big]\leq\left[\mu_0^\delta(\R)\right]^n\E\left[\big({\mathcal H}^c_b(M)\big)^n\right],
\end{align}
where we have set
\begin{align}\label{def: H_b^c(M)}
{\mathcal H}_b^c(M)=\int_{([-M, M]^d)^c}\left(\int_0^b
e^{i\xi\cdot B_s}ds\right)\left(\int_0^b
e^{-i\xi\cdot \widetilde{B}_s}ds\right)\mu(d\xi).
\end{align}

We now wish to apply subadditivity properties of ${\mathcal{H}}_b^c(M)$, such as \eqref{sub-2}, in order to obtain relation  \eqref{space cutoff}. A first step in this direction is to apply Chebyshev's inequality, which asserts that for all $\epsilon, b>0$ and  $k\geq 1$ we have
\begin{align}\label{space cutoff midstep}
(\epsilon^2b^2N)^{2k}\,\PP\bigg\{|\bar{Z}_b^{\delta,M}|^{1/2}\ge \epsilon b\bigg\}
\le N^{2k}
\E\Big[|\bar{Z}_b^{\delta, M}|^{2k}\Big]=N^{2k}
\E\Big[(\bar{Z}_b^{\delta, M})^{2k}\Big],
\end{align}
where we have introduced an additional parameter $N\geq 0$ to be specified later on.
We sum inequality \eqref{space cutoff midstep}  over $k$ and resort to the elementary inequality 
$$\frac{1}{2}\exp(x)\leq \frac{e^x+e^{-x}}{2}= \sum_{k=0}^\infty {x^{2k}}/{(2k)!},
$$ 
in order to get
\begin{align*}
\frac{1}{2}e^{\epsilon^2b^2N}\PP\left\{|\bar{Z}_b^{\delta,M}|^{1/2}\geq \epsilon b\right\}&\leq\sum_{k=0}^\infty\frac{(\epsilon^2b^2N)^{2k}}{(2k)!}\PP\left\{|\bar{Z}_b^{\delta,M}|^{1/2}\geq \epsilon b\right\}
\leq\sum_{k=0}^\infty\frac{N^{2k}}{(2k)!}\E\Big[(\bar{Z}_b^{\delta, M})^{2k}\Big].
\end{align*}
Therefore invoking the fact that $\E\big[(\bar{Z}_b^{\delta, M})^{k}\big]\geq 0$ for all $k\geq 0$, we have
\begin{align*}
\frac{1}{2}\exp\{\epsilon^2b^2N\}\PP\left\{|\bar{Z}_b^{\delta,M}|^{1/2}\geq \epsilon b\right\}
&\leq\sum_{k=0}^\infty\frac{N^{k}}{k!}\E\Big[(\bar{Z}_b^{\delta, M})^{k}\Big],
\end{align*}
and owing to inequality \eqref{eq: bound bar Z^delta M by H}, the above becomes
\begin{align*}
\frac{1}{2}\exp\{\epsilon^2b^2N\}\PP\left\{|\bar{Z}_b^{\delta,M}|^{1/2}\geq \epsilon b\right\}
&\leq\sum_{k=0}^\infty\frac{(N\mu_0^\delta(\R))^{k}}{k!} \E\left[\big({\mathcal H}^c_b(M)\big)^k\right],
\end{align*}
where we recall that ${\mathcal H}_b^c(M)$ is defined by \eqref{def: H_b^c(M)}.
Summarizing our considerations for this step,  we have found that for all $\epsilon, b>0$ and $N>0$ we have
\begin{align*}
\PP\left\{|\bar{Z}_b^{\delta,M}|^{1/2}\geq \epsilon b\right\}
&\leq{2}\exp\{-\epsilon^2b^2N\}\sum_{k=0}^\infty\frac{(N\mu_0^\delta(\R))^{k}}{k!} \E\left[\big({\mathcal H}^c_b(M)\big)^k\right]\\
&={2}\exp\{-\epsilon^2b^2N\}\sum_{k=0}^\infty\frac{(bN\mu_0^\delta(\R))^{k}}{k!} \left\{\frac{1}{b^k}\,\E\left[\big({\mathcal H}^c_b(M)\big)^k\right]\right\}.
\end{align*}

We can now apply Lemma \ref{sub} in the following way: we set $\theta=bN\mu_0^\delta(\R)$ and we assume that $b$ is an integer (generalizations to an arbitrary positive $b$ are left to the reader). Then iterating \eqref{sub-2} $b$ times and writing $R=Nb$ we end up with
\begin{align}\label{space cutoff midstep 1}
\PP\left\{|\bar{Z}_b^{\delta,M}|^{1/2}\geq \epsilon b\right\}\leq {2}\exp\{-\epsilon^2bR\}\left(\mathcal{A}^\delta(M,\lastchange{R})\right)^b,
\end{align}
where the quantity $\mathcal{A}^\delta(M,\lastchange{R})$ is defined by
\begin{equation}\label{c7}
\mathcal{A}^\delta(M,\lastchange{R})=\sum_{k=0}^\infty\frac{(R\mu_0^\delta(\R))^{k}}{k!} \E\left[\big({\mathcal H}^c_1(M)\big)^k\right].
\end{equation}
Finally, one can use the dominated convergence theorem to  show that the \lastchange{right-hand side} of relation~\eqref{space cutoff midstep 1} satisfies, \lastchange{for every $R>0$ small enough,}
\begin{align}\label{eq: dct result}
\lim_{M\to\infty}\mathcal{A}^\delta(M,\lastchange{R})=1.
\end{align}
Indeed, first note that the bound (3.1) in \cite{Ch18}, already used for our relation \eqref{eq: nth moment of bar Z_b^deta 4},  can be extended to the case when $H_0=1$ (i.e., the setting without time dependence). A direct application of this bound (or generalization of \eqref{eq: nth moment of bar Z_b^deta 4} to $\alpha_0=\hat{\alpha}_0=0$ and $b=1$) yields that for all $k\geq 1$ we have
$$\E \big[\big({\mathcal H}_1^c(M)\big)^k\big]\leq\E \big[\big({\mathcal H}_1(\R)\big)^k\big]\leq C^k k!,$$
where we have set
\begin{align*}
{\mathcal H}_1(\R)=\int_{\R}\left(\int_0^b
e^{i\xi\cdot B_s}ds\right)\left(\int_0^b
e^{-i\xi\cdot \widetilde{B}_s}ds\right)\mu(d\xi).
\end{align*}
Hence one can choose $R$ small enough such that the following domination of the general term of \eqref{eq: dct result} holds true,
\begin{align}\label{space cutoff midstep 2}
\mathcal{A}^\delta(M,\lastchange{R})
\leq
\sum_{k=0}^\infty\frac{(R\mu_0^\delta(\R))^{k}}{k!} \E\left[\big({\mathcal H}_1(\R)\big)^k\right]
\leq
\sum_{k=0}^\infty \lp C R\mu_0^\delta(\R)  \rp^{k}
<\infty.
\end{align}
Moreover, invoking relation \eqref{def: H_b^c(M)} it is readily checked that the $k$-th moment of $\mathcal{H}_1^c(M)$ in~\eqref{c7} is such that
$$\E \big[\big({\mathcal H}_1^c(M)\big)^k\big]=\int_{([-M,M]^d)^c)^k}\Big|\int_{[0,1]^k}\E\Big[\prod_{i=1}^k e^{i\xi_k\cdot B_{s_k}}\Big] ds\Big|^2 \mu(d\xi).$$
Thus the mapping $M\mapsto \E \big[\big({\mathcal H}_1^c(M)\big)^k\big]$ is monotone and decreasing.  
As a direct consequence, relation \eqref{eq: dct result} follows by dominated convergence.

Let us now turn to our partial objective \eqref{space cutoff}. Namely recast relation \eqref{space cutoff midstep 1} as
$$\frac{1}{b}\log\PP\left\{|\bar{Z}_b^{\delta,M}|^{1/2}\geq \epsilon b\right\}\leq \frac{\log 2}{b}-\epsilon^2 R+\log \mathcal{A}^\delta(M,\lastchange{R}).$$
In the \lastchange{right-hand side} above, one can fix $R>0$ small enough and then take $M$ large enough so that 
$$\log \mathcal{A}^\delta(M,\lastchange{R})\leq {\epsilon^4}.$$
Therefore for $\epsilon$ small enough we get
$$\frac{1}{b}\log\PP\left\{|\bar{Z}_b^{\delta,M}|^{1/2}\geq \epsilon b\right\}\leq \frac{\log 2}{b}-\frac{1}{2}\epsilon^2R,$$
from which \eqref{space cutoff} is easily deduced. 

Summarizing our considerations from this step, having \eqref{c-6}, \eqref{decomp: space cutoff} and \eqref{space cutoff} in mind we get that our claim \eqref{c-5} is achieved as soon as we can prove 
\begin{align}\label{last probability deduction}
\limsup_{b\to\infty}\frac{1}{b}\log\PP\left\{|{Z}_b^{\delta,M}|^{1/2}\geq u b\right\}\leq -\frac{u^2}{\kappa^4},
\end{align}
where we recall that $Z_b^{\delta,M}$ is defined by \eqref{Z^delta,M}.

\smallskip
\noindent{\it Step 4: Expression in terms of the moments of $Z_b^{\delta,M}$.}  Thanks to a standard use of Chebyshev's inequality, relation \eqref{last probability deduction} is achieved as long as we can prove the following Gaussian type bound for $|Z_b^{\delta,M}|^{1/2}$,
\begin{align}\label{c-9}
\limsup_{b\to\infty}\frac{1}{b}\log
\E\lc \exp\lp\theta|Z_b^{\delta,M}|^{1/2}\rp \rc
\le \frac{\kappa^4\theta^2}{4},
\end{align}
which should be valid for all $\theta\geq 0$.  In order to separate the variables $\eta_b$ and $\tilde{\eta}_b$ in the definition \eqref{Z^delta,M} of $Z_b^{\delta,M}$ we first apply Cauchy-Schwarz inequality and then the elementary inequality $2\sqrt{ab}\leq a+b$ for all $a,b\geq0$. We get
\begin{multline*}
|Z_b^{\delta,M}|^{1/2}
\le {\frac12}
\Bigg\{\bigg(
\int_{\R\times[-M, M]^d}\vert\eta_b(\lambda,\xi)\vert^2\mu^\delta_0(d\lambda)
\mu(d\xi)\bigg)^{1/2} \\
+\bigg(
\int_{\R\times[-M, M]^d}\vert\tilde{\eta}_b(\lambda,\xi)\vert^2
\mu^\delta_0(d\lambda)
\mu(d\xi)\bigg)^{1/2}\Bigg\}.
\end{multline*}
Taking into account the fact that $\eta$ and $\tilde{\eta}$ are independent, this entails
\begin{align}\label{eq: moment estimate Z^delta,M midstep 1}
\E\lc\exp\left(\theta|Z_b^{\delta,M}|^{1/2}\right)\rc
\le\left\{\E\lc\exp\left(\frac{\theta}{2}
|X_b^{\delta,M}|^{1/2}\right)\rc\right\}^2,
\end{align}
where the random variable $X_b^{\delta,M}$ is defined by
\begin{align}\label{def: X_b^delta M}X_b^{\delta,M}=\int_{\R\times[-M, M]^d}\vert\eta_b(\lambda,\xi)\vert^2
\mu^\delta_0(d\lambda)\mu(d\xi)
.\end{align}
Putting together \eqref{c-9} and \eqref{eq: moment estimate Z^delta,M midstep 1} and setting $\theta=2\, \theta$, we are reduced to prove
\begin{align}\label{c-10}
\limsup_{b\to\infty}\frac{1}{b}\log
\E\lc\exp\left(\theta |X_b^{\delta,M}|^{1/2}\right)\rc
\le \frac{\kappa^4\theta^2}{2}.
\end{align}
In addition, condition \eqref{c-10} can be expressed in terms of the moments of $X_b^{\delta,M}$. Indeed, taking into account the fact that $X_b^{\delta,M}$ is a positive random variable, a direct application of \cite[Lemma 1.2.6]{Ch-bk} asserts that \eqref{c-10} is equivalent to the following property:
\begin{align}\label{s-3}
\limsup_{b\to\infty}\frac{1}{b}\log\sum_{n=0}^\infty \frac{\theta^n}{n!}
\Big(\E\lc (X_b^{\delta,M})^n\rc\Big)^{1/2}
\le {\kappa^4\theta^2}.
\end{align}
We will use the formulation \eqref{s-3} below in order to replace the Brownian motion $B$ in $X_b^{\delta,M}$ by a Ornstein-Uhlenbeck process.

\smallskip
\noindent{\it Step 5: Expression in terms of an Ornstein-Uhlenbeck process.} 
Similarly to a strategy borrowed from \cite{HLN},
we now introduce a family of $\R^d$-valued Ornstein-Uhlenbeck processes indexed by $\alpha>0$, denoted by $B^\alpha$. The process $B^\alpha$ solves the equation 
\begin{align}\label{eq: O-U process}
dB_t^\alpha=-\alpha B_t^\alpha dt+dB_t,
\end{align}
where $B$ is our standing $d$-dimensional Wiener process. Now notice that $B$ can also be seen as an Ornstein-Uhlenbeck type process of the form
$$dB_t=-\alpha B_tdt+dW_t,\quad\text{with}\quad W_t=B_t+\alpha\int_0^tB_sds.$$
Therefore, if we denote by $\PP^\alpha$ the law of $B^\alpha$, a standard application of Girsanov's theorem yields
\begin{align}\label{eq: Girsanov}
\frac{d\PP^{\alpha}}{d\PP}\bigg\vert_{[0,\tau]}
&=\exp\bigg(-\alpha\int_0^\tau B_s\cdot dB_s-\frac{\alpha^2}{2}\int_0^\tau
\vert B_s\vert^2ds\bigg).
\end{align}
In addition, It\^{o}'s formula applied to $B$ entails that for any $\tau>0$ we have
$$|B_\tau|^2=2\int_0^\tau B_s\cdot dB_s+\tau d.$$
Therefore the exponential term in \eqref{eq: Girsanov} can be recast as
\begin{align}\label{ldp-12}
\frac{d\PP^{\alpha}}{d\PP}\bigg\vert_{[0,\tau]}=\exp\bigg(\frac{1}{2}\left({\alpha}\tau d-{\alpha}\vert B_\tau\vert^2
-{\alpha^2}\int_0^\tau
\vert B_s\vert^2ds\right)\bigg).
\end{align}
In particular, it is readily checked that
\begin{align}\label{ldp-13}
\frac{d\PP^{\alpha}}{d\PP}\bigg\vert_{[0,\tau]}\le \exp\Big(\frac{1}{2}\alpha  \tau d
\Big).
\end{align}
Let us also recall a moment comparison inequality which is obtained in \cite[relation (6.20)]{Ch18}. Namely for $n\geq1$ and the random variables $X_b^{\delta,M}$ defined by \eqref{def: X_b^delta M} we have
$$
\begin{aligned}
\E\big[(X_b^{\delta,M})^n\big]\leq \E^\alpha\big[(X_b^{\delta,M})^n\big],
\end{aligned}
$$
for all $\alpha>0$, where $\E^\alpha$ denotes the expectation under $\PP^\alpha$. Hence in order to prove \eqref{s-3} it will be enough to show a uniform bound in $\alpha$, namely
\begin{align*}
\limsup_{\alpha\to0}\limsup_{b\to\infty}\frac{1}{b}\log\sum_{n=0}^\infty \frac{\theta^n}{n!}
\Big(\E^\alpha\lc (X_b^{\delta,M})^n\rc\Big)^{1/2}
\le {\kappa^4\theta^2}.
\end{align*}
Therefore, invoking again the equivalence between \eqref{c-10} and \eqref{s-3} for positive random variables, we are reduced to prove
\begin{align}\label{reduction to O-U}
\limsup_{\alpha\to 0}\limsup_{b\to\infty}\frac{1}{b}\log
\E^\alpha\left[
\exp\left({\theta}|X_b^{\delta,M}|^{1/2}\right)\right]
\le \frac{\kappa^4\theta^2}{2}
\end{align}
for every $\theta>0$.  We now focus on inequality \eqref{reduction to O-U} for a generic $\theta>0$.

\smallskip
\noindent{\it Step 6: Space-time cutoff for the Ornstein-Uhlenbeck process.}  Let us introduce an additional parameter $N>0$ and write
$$X_b^{\delta,M}=X_b^{\delta,M,N}+\bar{X}_b^{\delta,M,N},$$
with
\begin{align}\label{def: X_b^delta M N}
&X_b^{\delta,M,N}=\int_{[-N,N]\times[-M,M]^d}|\eta_b(\lambda,\xi)|^2\mu_0^\delta(d\lambda)\mu(d\xi),\\
&\bar{X}_b^{\delta,M,N}=\int_{[-N,N]^c\times[-M,M]^d}|\eta_b(\lambda,\xi)|^2\mu_0^\delta(d\lambda)\mu(d\xi).\nonumber
\end{align}
Trivially bounding the oscillating exponential terms by $1$ in the definition \eqref{def: eta_b} of $\eta_b(\lambda,\xi)$, we obtain
\begin{align}\label{eq:bound bar X_b^delta M N}|\bar{X}_b^{\delta,M,N}|\leq b^2\mu_0^\delta([-N,N]^c)\mu([-M,M]^d).\end{align}
Therefore thanks to the fact that $\delta>0$, the \lastchange{right-hand side} of \eqref{eq:bound bar X_b^delta M N} can be made as small as desired by picking $N$ large enough. Similarly to what has been done in Step 2, one can thus prove an inequality of the same form as \eqref{s-1} for $\bar{X}_b^{\delta,M,N}$. We are now reduced to show that \eqref{reduction to O-U} holds true with $X_b^{\delta,M}$ replaced by $X_b^{\delta,M,N}$.

Next we consider a new parameter $K>0$ and we fix a value $\alpha>0$. We further decompose $X_b^{\delta,M,N}$ and write
\begin{align}\label{decomp X_b^delta M B}
X_b^{\delta,M,N}=X_b^{\delta,M,N} \1_{\Omega_{b,K}}+X_b^{\delta,M,N} \1_{\Omega_{b,K}^c},
\end{align}
where
\begin{equation*}
\Omega_{b, K}=\bigg\{\frac{1}{b}\int_0^b\vert B_s\vert ds\, {\le}\, K\bigg\}.
\end{equation*}
We will prove that for $K$ large enough, the quantity $X_b^{\delta,M,N}\1_{\Omega_{b,K}^c}$ is negligible with respect to $X_b^{\delta,M,N}\1_{\Omega_{b,K}}$. Indeed, resorting to  (\ref{ldp-12}) we have,
$$
\begin{aligned}
&\E^\alpha\left[\exp\bigg(\int_0^b\vert B_s\vert ds\bigg)\right]
\le\E\exp\bigg\{\int_0^b\Big(\vert B_s\vert
-\frac{\alpha^2}{2}\vert B_s\vert^2\Big)ds +\frac{\alpha d}{2}b\bigg\}.
\end{aligned}
$$
Now we can use the elementary inequality $x-\alpha^2x^2/2\leq 1/2\alpha^2$, valid for all $x>0$, in order to get
$$
\begin{aligned}
\lastchange{\E^\alpha\left[\exp\bigg(\int_0^b\vert B_s\vert ds\bigg)\right]}
\le  e^{C_\alpha b},\quad\text{where}\quad C_\alpha=\frac{1}{2\alpha^2}+\frac{\alpha d}{2}.
\end{aligned}
$$
Hence given $l>0$ one can choose $K>0$ large enough (say $K=l+2C_\alpha$) such that uniformly in $b$ we have
\begin{align}\label{est Omega}
\PP^\alpha(\Omega_{b,K}^c)
\le e^{-lb},
\end{align}
where we recall that $\Omega_{b,K}$ is defined by \eqref{decomp X_b^delta M B}.  In addition, one can trivially bound the quantity $|\eta_b(\lambda,\xi)|^2$ in the definition \eqref{def: X_b^delta M N} of $X_b^{\delta,M,N}$ by $b^2$. This yields
\[
\E^\alpha\left[\exp\left(\theta|X_b^{\delta,M,N}|^{1/2}\right)\1_{\Omega_{b,K}^c}\right]\leq \exp\left(C_{\delta,M,N}^{1/2}\theta b-l b\right),
\]
where $l$ above can be made arbitrarily large. In conclusion of this step, it is enough to prove that for all $\theta, \delta, M, N>0$, we have
\begin{align}\label{c-11}
\limsup_{\alpha\to 0}\limsup_{b\to\infty}\frac{1}{b}\log
\E^\alpha
\left[\exp\left({\theta}
|X_b^{\delta,M,N}|^{1/2}\right)\1_{\Omega_{b,K}}\right]
\le \frac{\kappa^4\theta^2}{2}.
\end{align}

\smallskip
\noindent{\it Step 7: Linearization procedure.}  One way to recast the definition \eqref{def: X_b^delta M N} of $|X_b^{\delta,M,N}|^{1/2}$ is to write 
\begin{align}\label{eq: X in terms of G norm}
|X_b^{\delta,M,N}|^{1/2}=\|\eta_b\|_{\mathcal G},
\end{align}
where the functional space $\cg$ is defined by
\begin{align}\label{def: G}
{\mathcal G}=\left\{ h\in L^2([-N,N]\times[-M,M]^d; \mu_0^\delta\otimes\mu); h(-\lambda,-\xi)=\bar{h}(\lambda,\xi)\right\}.\end{align}
In this step we show how to linearize the ${\mathcal G}$-norm above when $\omega\in\Omega_{b,K}$. To this aim, going back to \eqref{def: eta_b}, notice that for $\lambda_1, \lambda_2\in\R$ and $\xi_1,\xi_2\in\R^d$ we have
\begin{align*}
&|\eta_b(\lambda_2,\xi_2)-\eta_b(\lambda_1,\xi_1)|
\leq\int_0^b|e^{\ii(b^{-1}\lambda_2+\xi_2\cdot B_s)}-e^{i(b^{-1}\lambda_1+\xi_1\cdot B_s)}|ds\\
&\leq |\lambda_2-\lambda_1|+\left(\int_0^b|B_s|ds\right)|\xi_2-\xi_1|
\leq|\lambda_2-\lambda_1|+Kb|\xi_2-\xi_1|,
\end{align*}
where we have invoked the fact that $\omega\in\Omega_{b,K}$, with  $\Omega_{b,K}$ given by \eqref{decomp X_b^delta M B}, for the last inequality. Therefore, setting
\begin{multline}
{\mathcal C}=\Big\{h(\cdot,\cdot)\in {\mathcal G};\hskip.1in 
\vert h(\lambda, \xi)\vert\le 1,\ 
\vert h(\lambda_1,\xi_1)-h(\lambda_2,\xi_2)\vert
\le (1+K)\vert (\lambda_1,\xi_1)-(\lambda_2,\xi_2)\vert\Big\},
\end{multline}
it is easily seen that $b^{-1}\eta_b\in{\mathcal C}$ for all $b\geq1$, whenever $\omega\in\Omega_{b,K}$. Moreover, some standard uniform continuity arguments show that the closure ${\mathcal K}$ of ${\mathcal C}$ in ${\mathcal G}$ is a compact set.  We have thus proved that for all $b\geq1$ and $\omega\in\Omega_{b,K}$ we have
$$b^{-1}\eta_b\in{\mathcal K}.$$

Our next step will be to construct a finite cover of ${\mathcal K}$. To this aim consider the unit ball in ${\mathcal G}$, denoted by $B_{\mathcal G}(0,1)$. We also consider the following set for any $f\in B_{\mathcal G}(0,1)$:
\begin{align}\label{def: O_f}
{\mathcal O}_f=\{h\in {\mathcal G};\hskip.1in 
\langle f, h\rangle_{\mathcal G}> \|h\|_{\mathcal G}-\epsilon\}.
\end{align}
Then we claim that the family $\{{\mathcal O}_f; f\in B_{\mathcal G}(0,1)\}$ covers ${\mathcal K}$. 
Indeed, consider $f\in B_{\mathcal G}(0,1)$ such that $\|f\|_{\cg}=1$ and a constant $a>0$. We set $h=a f$. Then it is readily checked that $\langle f, \, h\rangle_{\cg}=\|h\|_{\cg}$. Therefore we get
\begin{equation*}
\ck
\subset
\lcl  a f ; \, \|f\|_{\cg}=1 , \, a >0 \rcl
\subset
\lcl {\mathcal O}_f; f\in B_{\mathcal G}(0,1) \rcl .
\end{equation*}
Otherwise stated, $\{{\mathcal O}_f; f\in B_{\mathcal G}(0,1)\}$ covers ${\mathcal K}$. In the sequel we will extract a finite family $\{{\mathcal O}_{f_i}, 1\leq i\leq m\}$ which still covers $\mathcal K$, which is possible since $\mathcal K$ is a compact set.

Let us now go back to the random function $b^{-1}\eta_b$, seen as an element of $\mathcal K$. From the previous consideration we know that for almost every $\omega\in\Omega_{b,K}$ there exists $j=j(\omega)\in\{1,\ldots ,m\}$ such that $\lastchange{b^{-1}}\eta_b\in{\mathcal O}_{f_j}$. Hence from the very definition \eqref{def: O_f} of ${\mathcal O}_{f_j}$, we get
$$
\|b^{-1}\eta_b(\cdot, \cdot)\|_{\lastchange{\mathcal G}}\le \epsilon+ 
\max_{1\le j\le m}\langle f_j, {b^{-1}\eta_b}\rangle_{\lastchange{\mathcal G}}
$$
Plugging this inequality into \eqref{eq: X in terms of G norm} we end up with
\begin{multline*}
\E^\alpha\left[
\exp\left({\theta}
|X_b^{\delta,M,N}|^{1/2}\right)\1_{\Omega_{\lastchange{b},K}}\right]
=
\E^\alpha\left[\exp\left({\theta}
\|\eta_b\|_{\mathcal G}\right)\1_{\Omega_{\lastchange{b},K}}\right]\\
\leq\E^\alpha\left[\exp\left(\lastchange{\theta}\max_{1\leq j\leq m}
\langle f_j, \eta_b\rangle_{\mathcal G}+b\theta\epsilon\right)\right]  
\le {e^{b\theta\epsilon}} \sum_{j=1}^m\E^\alpha\left[\exp\left({\theta}
\langle f_j, \eta_b\rangle_{\mathcal G}\right)\right].
\end{multline*}
Hence some elementary properties of the logarithmic function (see e.g \cite[Lemma 1.2.15]{DZ}) entail
\begin{multline}\label{final reduction}
\limsup_{b\to\infty}\frac{1}{b}\log \E^\alpha\left[
\exp\left({\theta}
|X_b^{\delta,M,N}|^{1/2}\right)\1_{\Omega_{\lastchange{b},K}}\right]\\
\le {\theta\epsilon} +\max_{1\le j\le m}\left\{
\limsup_{b\to\infty}\frac{1}{b}
\log\left(\E^\alpha\lc \exp\left({\theta} \langle f_j, \eta_b\rangle_{\mathcal G}\rc \right)\right)\right\}.
\end{multline}
We will now treat each term in the \lastchange{right-hand side} of \eqref{final reduction} separately.

\smallskip
\noindent{\it Step 8: Feynman-Kac type asymptotics.} Due to the fact that $\mathcal G$ is a subspace of $L^2([-N,N]\times[-M,M]^d: \mu_0^\delta\otimes\mu)$, for $j=1,\ldots ,m$ we have
$$
\begin{aligned}
\langle f_j, \eta_b\rangle_{\mathcal G}&=\int_{[-N, N]\times[-M, M]^d}
f_j(\lambda,\xi){\bar{\eta}_b(\lambda,\xi)}\lastchange{\mu_0^\delta}(d\lambda)\mu(d\xi).
\end{aligned}
$$
Recalling once again the definition \eqref{def: eta_b} of $\eta_b$, this yields
$$
\begin{aligned}
\langle f_j, \eta_b\rangle_{\mathcal G}=\int_0^b\tilde{f}_{j}\Big(\frac{s}{b}, B_s\Big)ds,
\end{aligned}
$$
where the functions $\tilde{f}_{j}$ are given by  the following relation,
\begin{align}\label{def: bar f}
\tilde{f}_{j}(s, x)=\int_{[-N, N]\times[-M, M]^d}{f_j(\lambda, \xi)}
e^{\ii\left(\lambda s+\xi\cdot x\right)}
\lastchange{\mu_0^\delta}(d\lambda)\mu(d\xi).
\end{align}
Also notice that according to the definition \eqref{def: G} of $\mathcal G$ we have $f_j(\lambda, \xi)=\lastchange{\overline{f_{j}}}(-\lambda, -\xi)$ for all $j=1,\ldots ,m, \lambda\in\R$ and $\xi\in\R^d$. With the expression \eqref{def: bar f} in mind, this yields that $\tilde{f}_{j}$ is real-valued. Thus inequality (\ref{ldp-13}) entails
\begin{align}\label{new-2}
\E^\alpha\lc \exp\left({\theta}\langle f_j, \eta_b\rangle_{\mathcal G}\right) \rc
\le \exp\Big(\frac{\lastchange{\al} d}{2}b\Big)\E\left[\exp\bigg({\theta}
\int_0^b\tilde{f}_{j}\Big(\frac{s}{b}, B_s\Big)ds\bigg)\right],
\end{align}
where the expectation on the \lastchange{right-hand side} is now taken with respect to a $\R^d$-valued Brownian motion $B$.

With \eqref{new-2} in hand, we are now back to a more classical Feynman-Kac computation. Specifically, $\tilde{f}_{j}$ can be seen as the Fourier transform of a finite and compactly supported measure, as is apparent from \eqref{def: bar f}. Therefore it fulfills all the regularity assumptions allowing to apply \cite[Proposition 3.1]{CHSX}. Applying this proposition we thus get
\begin{align}\label{new-3}
&\lim_{b\to\infty}\frac{1}{b}\log\E\left[\exp\bigg({\theta}
\int_0^b\tilde{f}_{j}\Big(\frac{s}{b}, B_s\Big)ds\bigg)\right]\\
&=\sup_{g\in {\mathcal A}_d}\bigg\{{\theta}\int_0^1\!\!\int_{\R^d}\tilde{f}_{j}(s,x)
g^2(s,x)dxds
-\frac12\int_0^1\!\!\int_{\R^d}\vert\nabla_xg(s,x)\vert^2dxds\bigg\},\nonumber
\end{align}
where the space ${\mathcal A}_d$ is introduced in Definition \ref{def: A_d}.

\smallskip
\noindent{\it Step 9: Evaluation of the Feynman-Kac asymptotics.} Let us analyze the \lastchange{right-hand side} of relation~\eqref{new-3}. Owing to the definition of $\tilde{f}_{j}$ in \eqref{def: bar f}, it is readily checked that
$$\int_0^1\!\!\int_{\R^d}\tilde{f}_{j}(s,x)
g^2(s,x)dxds=\langle f_j,\tilde{\mathcal F}(g^2)\rangle_{\mathcal G}
,$$
where we recall that the space $\mathcal G$ is given by \eqref{def: G}, and where we have set
\begin{align}\label{tilde cal F}
\tilde{\mathcal F}(g^2)(\lambda,\xi)=\int_0^1\int_{\R^d}g^2(s,x)e^{-\ii(\lambda s+\xi\cdot x)}dxds.
\end{align}
 Furthermore, each $f_j$ sits in the unit ball $B_{\mathcal G}(0,1)$ and thus
$$\left|\int_0^1\!\!\int_{\R^d}\tilde{f}_{j}(s,x)
g^2(s,x)dxds\right|\le\|{\tilde{\mathcal F}}(g^2)\|_{\mathcal G}.$$

Now it is easily seen that
$$
\begin{aligned}
\|{\tilde{\mathcal F}}(g^2)\|^2_{\mathcal G}
&=\int_{[-N, N]\times[-M, M]^d}\vert{\tilde{\mathcal F}}(g^2)\lastchange{(\la,\xi)}\vert^2
\lastchange{\mu_0^\delta}(d\lambda)\mu(d\xi)\\
&\le \int_{\R^{d+1}}\vert{\tilde{\mathcal F}}(g^2)\lastchange{(\la,\xi)}\vert^2
\mu_0(d\lambda)\mu(d\xi)\\
&=\int_{[0,1]^2}\!\!\int_{\R^d\times\R^d}\gamma_0(s-r)\gamma(x-y)
g^2(x)g^2(y)dxdydrds.
\end{aligned}
$$
%where we recall that $\mathcal H$ is the Cameron-Martin space introduced in \eqref{eq:covariance-W-with-Fourier}. 
Plugging this information into \eqref{new-3}, we have obtained that
\begin{align}\label{eq: step 9 midstep}
&\lim_{b\to\infty}\frac{1}{b}\log\E\left[\exp\bigg({\theta}
\int_0^b\tilde{f}_{j}\Big(\frac{s}{b}, B_s\Big)ds\bigg)\right] \\
&\leq\sup_{g\in {\mathcal A}_d}\bigg\{{\theta}\bigg(\int_{[0,1]^2}\!\!\int_{\R^d\times\R^d}\gamma_0(s-r)\gamma(x-y)
g^2(x)g^2(y)dxdydrds\bigg)^{1/2} \notag\\
&\hspace{2.5in}-\frac12\|\nabla_xg(s,x)\|^2_{L^2([0,1]\times\R^d)}\bigg\}.\nonumber
\end{align}
Now resorting to \eqref{eq:funct-ineq}, the \lastchange{right-hand side} of~\eqref{eq: step 9 midstep} can be upper bounded by
\begin{multline*}
\sup_{g\in{\mathcal A_{d}}}\left\{\theta\kappa^2\|\nabla_xg\|_{L^2([0,1]\times\R^d)}-\frac{1}{2}\|\nabla_xg\|^2_{L^2([0,1]\times\R^d)}\right\} 
=\sup_{z>0}\left\{\theta\kappa^2z-\frac{1}{2}z^2\right\}=\frac{1}{2}\kappa^4\theta^2,
\end{multline*}   
where the last equality stems from a trivial optimization procedure.
Summarizing our computations for this step, we have found that for $j=1,\ldots,m$ we have
\begin{align}\label{eq: F-K asymptotics}
\lim_{b\to\infty}\frac{1}{b}\log\E\left[\exp\bigg(\theta\int_0^b\tilde{f}_{j}\Big(\frac{s}{b}, B_s\Big)ds\bigg)\right]\leq \frac{1}{2}\kappa^4\theta^2.
\end{align}

\smallskip
\noindent{\it Step 10: Conclusion.}  Recall that our desired upper bound has been successively reduced to~\eqref{c-11}. Moreover, we have seen that the \lastchange{left-hand side} of \eqref{c-11} is bounded by the \lastchange{right-hand side} of \eqref{final reduction} for any arbitrary small $\epsilon$. Therefore, we are left with the evaluation of 
\[
A_{\theta, \epsilon,\alpha}\equiv\theta\epsilon+\max_{1\le j\le m}\left\{
\limsup_{b\to\infty}\frac{1}{b}\log\left(\E^\alpha\left[\exp\left({\theta}
\langle f_j, \eta_b\rangle_{\mathcal G}\right)\right]\right)\right\}.
\]
Now putting together \eqref{new-2} and \eqref{eq: F-K asymptotics} we obtain that 
\[
A_{\theta,\epsilon,\alpha}\leq \theta\epsilon+\frac{\alpha d}{2}+\frac{1}{2}\kappa^4\theta^2.
\]
Since $\epsilon$ and $\alpha$ can be made arbitrarily small, we have shown that relation \eqref{c-11} holds true. This finishes the proof of relation \eqref{eq: moment bound alpha^12}. With~\eqref{eq: moment bound alpha^12} in hand, Theorem \ref{thm:explosion-critical}-\ref{critical-i} is obtained along the same lines as for Proposition~\ref{prop:cvgce-FK-Skorohod}.
\qed

\begin{remark}\label{rm: L^2 solution existence}
According to Remark \ref{rmk:critical-sko}-\ref{it:sko-eq-limit}, we could prove that the limit of $u^{\ep,\di}$ in Theorem~\ref{thm:explosion-critical}-\ref{critical-i} also solves the mild Skorohod equation~\eqref{eq:she-sko}. Moreover, thanks to a slight elaboration of Proposition~\ref{prop:cvgce-FK-Skorohod} we could also obtain the uniqueness of the solution. We have not included those details for sake of conciseness.
\end{remark}

\subsubsection{Convergence for $p>2$}  Recall that we have proved item \ref{critical-i} in Theorem \ref{thm:explosion-critical}. Namely we have shown the $L^{2}$ convergence of $u_{t}^{\ep,\di}$ to the solution $u_{t}^{\di}$ of \eqref{eq:she-sko} (see Remark \ref{rm: L^2 solution existence} about a notion of solution)
for $t<t_0=t_0(2)$. In this section we extend this result to a general $p>2$, that is we prove item \ref{critical-ii} in Theorem~\ref{thm:explosion-critical}. As in Section \ref{sec:convergence-p-equal-2}, we will first focus on showing that $\mathbf{E}[|\lastchange{u^\di_t}(x)|^p]$ is finite whenever $t<t_{0}(p)$.

In order to prove the $L^{p}$-boundedness of $\lastchange{u^\di_{t}}(x)$, we introduce an additional intensity parameter $\zeta>0$ and consider the process $\lastchange{u^{(\zeta)}}$, solution of the following slight extension of~\eqref{eq:she-sko}:
\begin{align*}
\partial_t \lastchange{u^{(\zeta)}_t}(x)=\frac{1}{2}\Delta \lastchange{u^{(\zeta)}}(x)+\sqrt{\zeta} \, \lastchange{u^{(\zeta)}}(x)\diamond\dot{W}_t(x),
\end{align*}
with initial condition $u_0(x)=1$. According to an hypercontractivity inequality shown in \cite[Theorem 1]{Le}, the following inequality holds true for all $p\geq 2$ and $(t,x)\in\R_+\times\R^d$
\begin{align}\label{eq: hypercontractivity}
\|\lastchange{u^{(\zeta)}_t}(x)\|_{L^p(\Omega)}\leq\|\lastchange{u^{((p-1)\zeta)}_t}(x)\|_{L^2(\Omega)},
\end{align}
or otherwise stated for $\zeta=1$,
\begin{align}\label{eq: hypercontractivity 2}
\mathbf{E}\left[|\lastchange{u^\di_t}(x)|^p\right]\leq\left(\mathbf{E}\left[|u_t^{(p-1)}(x)|^2\right]\right)^{p/2}.
\end{align}
In addition, writing formula \eqref{eq: FK moments-sko} for $p=2$ and for a noise $\sqrt{p-1}\,\dot{W}$, we get
$$
\begin{aligned}
\mathbf{E}\left[|u_t^{(p-1)}(x)|^2\right]=
\E\left[\exp\bigg((p-1)\int_{[0,t]^2}
\gamma_0(s-r)\gamma\big(B_s-\widetilde{B}_r\big)
dsdr\bigg)\right].
\end{aligned}
$$
Now resorting to a simple change of variables and invoking formula \eqref{eq: FK moments-sko} again, we end up with
\begin{align}\label{eq: L^2 computation}
\mathbf{E}\left[|u_t^{(p-1)}(x)|^2\right]
=&\E\left[
\exp\bigg(\int_{[ 0,{(p-1)^{1/(2H_0-1)}}t]^2}
\gamma_0(s-r)\gamma\big(B_s-\widetilde{B}_r\big)
dsdr\bigg)\right]\notag\\
=&\lastchange{\be}\left[ |\lastchange{u^\di_{(p-1)^{1/(2H_0-1)}t}}(x)|^2\right].
\end{align}
Gathering \eqref{eq: hypercontractivity 2} and \eqref{eq: L^2 computation}, we have thus obtained
\begin{align}\label{c12}
\mathbf{E}\left[|\lastchange{u^\di_t}(x)|^p\right]\leq \left(\lastchange{\be}\left[ |\lastchange{u^\di_{(p-1)^{1/(2H_0-1)}t}}(x)|^2\right]\right)^{p/2}.
\end{align}
Applying Theorem \ref{thm:explosion-critical} - item \ref{critical-i} to the right hand side of \eqref{c12}, we obtain that $\lastchange{u^\di_t}(x)\in L^p(\Omega)$ when $t<t_0(p)$.

We now turn to the proof of the fact that $u^{\epsilon,\di}_t(x)$ converges to $\lastchange{u^\di_t}(x)$ in $L^p(\Omega)$ for $t<t_0(p)$. To this aim, we fix a $p'>p$ such that $t<t_0(p')<t_0(p)$.  By the same argument that we used to get~\eqref{c12},  we obtain the following inequality for $u_t^{\epsilon,\di}(x)$,
\begin{align}\label{c13}
\mathbf{E}\left[\left| u_t^{\epsilon,\di}(x)\right|^{p'}\right]\leq 
\left(\be\left[ \left| u^{\epsilon,\di}_{(p'-1)^{1/(2H_0-1)}t}(x) \right|^2\right]\right)^{p'/2}.
\end{align}
Furthermore, it is easily seen from the proof of Proposition \ref{prop:cvgce-FK-Skorohod} that the $L^2(\Omega)$-norm of $u_t^{\epsilon,\di}(x)$ is dominated by that of $\lastchange{u^\di_t}(x)$. Therefore relation \eqref{c13} implies
\begin{align*}
\mathbf{E}\left[\left| u_t^{\epsilon,\di}(x) \right|^{p'}\right]\leq \left(\lastchange{\be}
\left[ \left |u^\di_{(p'-1)^{1/(2H_0-1)}t}(x)\right|^2\right]\right)^{p'/2}.
\end{align*}
Now by our choice of $p'$ (such that $t<t_0(p')$), we conclude
\begin{align*}
\sup_{\epsilon>0}\mathbf{E}\left[|u_t^{\epsilon,\di}(x)|^{p'}\right]<\infty.
\end{align*}
Therefore, the family $\{|u^{\epsilon,\di}_t(x)|^p\}_{\epsilon>0}$ is uniformly integrable.  This together with the fact that $u^{\epsilon,\di}_t(x)$ converges to $\lastchange{u^\di_t}(x)$ in $L^2(\Omega)$ gives us the convergence in $L^p(\Omega)$. The proof of  Theorem~\ref{thm:explosion-critical} - item \ref{critical-ii} is thus completed.\hfill$\Box$

\subsection{Proof of  Theorem \ref{thm:explosion-critical}  - item \ref{critical-iii} }
Recall that $u^{\ep,\di}_t(x)$ is defined in Proposition \ref{approx solution}. In the previous section, we essentially proved that $u^{\ep,\di}_t(x)$ converges to $\lastchange{u^\di_{t}}(x)$ in $L^p$ when $t<t_0(p)$. In this section, we show that $\|u^{\ep,\di}_t(x)\|_p$ diverges as $\epsilon\to0$ when $t>t_0(p)$.

In order to state a lower bound on the $L^p(\Omega)$ moments of $u^{\epsilon,\di}$, we first define a functional space ${\mathcal G}_0$ which generalizes the space $\mathcal G$ introduced in \eqref{def: G}. 

\begin{definition}\label{def: g_0}

Let $C_{0,b}(\R\times\R^d)$ be the space of compactly supported and bounded functions on $\R\times\R^d$. We set
\begin{multline*}
\mathcal{G}_0=\Big\{h\in L^2(\R\times\R^{d},\mu_0\otimes\mu)\cap C_{0,b}(\R\times\R^d); \
h(-\lambda,-\xi)={\bar{h}(\lambda,\xi)}, \  \mu_0\otimes\mu-\textrm{a.e.}\Big\}.
\end{multline*}
\end{definition}

We can now state a first lower bound on the moments of $u^{\epsilon,\di}$ in terms of a variational quantity.

\begin{proposition}\label{Lp representation}
For $\epsilon>0$ consider the solution $u^{\epsilon,\di}$ to the regularized Skorohod equation~\eqref{eq:heat-sko-approx}. We assume that the critical conditions \eqref{eq:sko-critical-regime} are met. Let $p>1$ and denote its conjugate exponent by $q$. Then for $t>0$ and $x\in\R^d$ we have
\begin{align}\label{var representation}
\liminf_{\ep\to0}\|u^{\ep,\di}_t(x)\|_p
\geq\sup_{h\in\mathcal{G}_0}\E\left[\exp\left(\int_0^t\tilde{h}(s,B_s)ds-\frac{q-1}{2}\int_{\R^{d+1}}|h(\lambda,\xi)|^2\mu_0(d\lambda)\mu(d\xi)\right)\right],
\end{align}
where the function $\tilde{h}$ is defined similarly to $\tilde{f}_j$ in \eqref{def: bar f}, namely,
\begin{align}\label{eq: def of tilde h}\tilde{h}(s,x)=\int_{\R\times\R^{d}}h(\lambda,\xi)e^{i(\lambda s+\xi\cdot x)}\mu_0(d\lambda)\mu(d\xi).\end{align}
\end{proposition}
\begin{proof}
Recall that $u^{\epsilon,\di}_t(x)$ is also given by expression \eqref{eq:approx-FK-sol-sko}, where $V_t^{\epsilon,B}(x)$ is introduced in \eqref{eq:def-Vt-epsilon-2}. Next we recall the definition of ${\mathcal S}$-transform on the Wiener space related to our noise $W$ (see \cite{HKP} for more details about the $\mathcal S$-transform). Having in mind the notation introduced in Section \ref{intro Wiener space}, the $\mathcal S$-transform of $F\equiv u_t^{\epsilon,\di}(x)$ is defined for $\varphi\in\mathcal{H}$ by
\[{\mathcal S}F(\varphi)=\mathbf{E}[F{\mathcal E}_\varphi],\]
where the martingale exponential ${\mathcal E}_\varphi$ is given by
\begin{align}\label{E_epsilon}
{\mathcal E}_\varphi=\exp\left(W(\varphi)-\frac{1}{2}\|\varphi\|_{\mathcal H}^2\right).
\end{align}
As highlighted in \cite[Chapter 2]{HKP}, the $\mathcal S$-transform has to be considered as the equivalent of the Fourier transform on a Wiener space. However in our context we will just use the following basic estimate for $F=u^{\epsilon,\di}_t(x)$:
\begin{align}\label{eq: estimate S transform}
\|F\|_{p}\geq \sup\left\{\frac{{\mathcal S}F(\varphi)}{\|{\mathcal E}_\varphi\|_q}; \varphi\in{\mathcal H}\right\},
\end{align}
where we recall that $q$ is the conjugate of $p$.
We will now analyze the \lastchange{right-hand side} of relation~\eqref{eq: estimate S transform}.

In order to evaluate the $\mathcal S$-transform ${\mathcal S}F(\varphi)$ in \eqref{eq: estimate S transform}, We resort to the Feynman-Kac formula \eqref{eq:approx-FK-sol-sko} for $u^{\epsilon,\di}_t(x)$, Fubini's theorem and the isometry \eqref{eq:covariance-W-with-Fourier} on our standing Wiener space. Similarly to \eqref{a02}-\eqref{a101}, albeit with a time-space Fourier transform, we get 
\begin{align}\label{eq: S transform of F}
{\mathcal S}F(\varphi)=\E\left[\exp\left(\int_{\R\times\R^d}\mathcal{F}\varphi(\lambda,\xi)\bar{{\mathcal F}}{\psi_t}(\lambda,\xi)\mu_0(d\lambda)\mu(d\xi)\right)\right],
\end{align}
where the function $\psi_t$ is given by
$$\psi_t(\tau,x)=\int_0^tp_\epsilon(\tau-(t-s),x-B_s)ds.$$
Evaluating the Fourier transform of $\psi_t$  and plugging into \eqref{eq: S transform of F}, we thus get
\begin{align}\label{eq: S transform of F 2}
{\mathcal S}F(\varphi)=\int_0^t\left[\int_{\R\times\R^d}e^{-\epsilon^2(\lambda^2+|\xi|^2)/2}e^{i(\lambda(t-s)+\xi\cdot B_t)}\mathcal{F}\varphi(\lambda,\xi)\mu_0(d\lambda)\mu(d\xi)\right]ds.
\end{align}

Let us now compute the quantity $\|{\mathcal E}_\varphi\|_q$ in \eqref{eq: estimate S transform}. Owing to the fact that $\lastchange{\|\mathcal{E}_g\|_1}=1$ for any $g\in\mathcal H$, we easily get that 
\begin{align}\label{eq: L^q norm of E_phi}
\|\mathcal{E}_\varphi\|_q=\exp\left(\frac{q-1}{2}\int_{\R\times\R^d}|\mathcal{F}\varphi(\lambda,\xi)|^2\mu_0(d\lambda)\mu(d\xi)\right).
\end{align}
Therefore gathering \eqref{eq: S transform of F 2} and \eqref{eq: L^q norm of E_phi} into \eqref{eq: estimate S transform}, 
taking limit $\epsilon\to0$ in \eqref{eq: S transform of F 2} and observing that $\mathcal{F}\varphi\in L^2(\R\times\R^d; \mu_0\otimes\mu)$ whenever $\varphi\in\mathcal{H}$, we end up with our claim \eqref{var representation}.

\end{proof}

\begin{remark}
By the local finiteness and the symmetry of $\mu_0$ and $\mu$, the function $\tilde{h}$ is real-valued, bounded and uniformly continuous on $[0,1]\times\R^d$ whenever $h\in\cg_{0}$.
\end{remark}

Our next step is to relate the lower bound \eqref{var representation} to the space ${\mathcal A}_d$ introduced in Definition~\ref{def: A_d}. This is summarized in the following lemma. 
\begin{lemma}\label{lemma: positivity imply no L^p norm}
Let us assume that the conditions of Proposition \ref{Lp representation} are met. We also suppose that the following condition is satisfied,
\begin{align}\label{eq: positivity imply no L^p norm}
\sup_{h\in\mathcal{G}_0}\sup_{g\in\mathcal{A}_d}\left\{\lastchange{\phi_t}(h,g)-\frac{1}{2}\int_0^1\int_{\R^d}|\nabla_xg(s,x)|^2dxds\right\}>0,
\end{align}
where we recall that $\mathcal{G}_0$ is the space introduced in Definition \ref{def: g_0} and where the variational quantity $\lastchange{\phi_t}(h,g)$ is defined by
\begin{multline}\label{eq:def-phi-h-g}
\lastchange{\phi_t}(h,g)=\int_{\R\times\R^d}h(\lambda,\xi)\tilde{\mathcal{F}}g^2(\lambda,\xi)\mu_0(d\lambda)\mu(d\xi)\\
-\frac{1}{2t^{2H_0-1}(p-1)}\int_{\R\times\R^d}|h(\lambda,\xi)|^2\mu_0(d\lambda)\mu(d\xi),
\end{multline}
with $\tilde{\mathcal F}g^2$ being the truncated Fourier transform given in \eqref{tilde cal F}.
Then we have that
\begin{align}\label{eq:blowup L^p norm}\lim_{\epsilon\to0}\|u^{\epsilon,\di}_t(x)\|_p=\infty.\end{align}
\end{lemma}

\begin{proof}%The key to the proof is to apply Proposition \ref{Lp representation} to a family of carefully chosen functions in $\mathcal{G}_0.$ For the sake of conciseness, we postpone the details to Appendix \ref{Proof of Lemma positive imply no L^p norm}. 
We start from the \lastchange{right-hand side} of \eqref{var representation} and consider a generic $h\in\mathcal{G}_0$. Then for $b, t>0$ we also introduce a family of rescaled functions $\{h_{t,b}; t,b>0\}$ given by
$$h_{t,b}(\lambda,\xi)=\frac{1}{t^{2H_0-1}}h\left(t\lambda,\left(\frac{t}{b}\right)^{1/2}\xi\right).$$
Each $h_{t,b}$ is an element of $\mathcal{G}_0$ and  we will now evaluate the expression \eqref{var representation} for those functions.  We will take the next two observations into account:

\begin{enumerate}[wide, labelwidth=!, labelindent=0pt, label=(\roman*)]
\setlength\itemsep{.01in}
\item
The function $\tilde{h}_{t,b}$ defined by \eqref{eq: def of tilde h} can be computed thanks to an elementary change of variable. Owing to an additional Brownian scaling argument, we get
\begin{align}\label{eq: term 1}
\int_0^t\tilde{h}_{t,b}(s, B_s)ds\stackrel{(d)}{=}\int_0^b\tilde{h}\left(\frac{s}{b}, B_s\right)ds,
\end{align}

\item
Under our standing assumption $d-H=1$ we also have
\begin{align}\label{eq: term 2}
\int_{\R\times\R^{d}}|h_{t,b}(\lambda,\xi)|^2\mu_0(d\lambda)\mu(d\xi)=\frac{b}{t^{2H_0-1}}\int_{\R\times\R^{d}}|h(\lambda,\xi)|^2\mu_0(d\lambda)\mu(d\xi).
\end{align}

\end{enumerate}
Therefore plugging \eqref{eq: term 1} and \eqref{eq: term 2} into \eqref{var representation}, we get that 
\begin{align}
&\liminf_{\epsilon\to0}\|u^{\epsilon,\di}_t(x)\|_p\nonumber\\
&\quad\geq \sup_{b>0}\sup_{h\in\mathcal{G}_0}\E\left[\exp\left(\int_0^b\tilde{h}\left(\frac{s}{b},B_s\right)ds-\frac{(q-1)b}{2t^{2H_0-1}}\int_{\R\times\R^d}|h(\lambda,\xi)|^2\mu_0(d\lambda)\mu(ds)\right)\right].\label{var representation 2}
\end{align}
We are now in a position to apply \cite[Proposition 3.1]{Ch18} and take limits as $b\to\infty$ in \eqref{var representation 2}. Indeed it is readily checked that $\tilde{h}$ is bounded and uniformly continuous whenever $h\in\mathcal{G}_0$. Hence \cite[Proposition 3.1]{Ch18} reads 
\begin{align}\label{eq: term 3}
&\lim_{b\to\infty}\frac{1}{b}\log\E\left[\exp\left(\int_0^b\tilde{h}\left(\frac{s}{b}, B_s\right)ds\right)\right]\nonumber\\
=&\sup_{g\in\mathcal{A}_d}\left\{\int_0^1\int_{\R^d}\tilde{h}(s,x)g^2(s,x)dxds-\frac{1}{2}\int_0^1\int_{\R^d}|\nabla_xg(s,x)|^2dxds\right\}.
\end{align}
We also trivially have
\begin{align}
&\lim_{b\to\infty}\frac{1}{b}\log\left(\exp\left(\frac{(q-1)b}{2t^{2H_0-1}}\int_{\R\times\R^d}|h(\lambda,\xi)|^2\mu_0(d\lambda)\mu(ds)\right)\right)\nonumber\\
&\quad =\frac{(q-1)}{2t^{2H_0-1}}\int_{\R\times\R^d}|h(\lambda,\xi)|^2\mu_0(d\lambda)\mu(ds)\label{var rep trivial term}
\end{align}
Gathering \eqref{eq: term 3} and \eqref{var rep trivial term} into \eqref{var representation 2} and considering the limit case $b\to\infty$, we have obtained that \eqref{eq:blowup L^p norm} holds true
%\begin{align*}
%\liminf_{\epsilon\to0}\|u^{\epsilon,\di}_t(x)\|_p=\infty
%\end{align*} 
whenever the following relation is met for at least a $h\in\mathcal{G}_0$:
\begin{multline}
\sup_{g\in\mathcal{A}_d}\bigg\{\int_0^1\int_{\R^{d}}\tilde{h}(s,x)g^2(s,x)dxds
-\frac{1}{2}\int_0^1\int_{\R^d}|\nabla_xg(s,x)|^2dxds  \\
-\frac{q-1}{2t^{2H_0-1}}\int_{\R\times\R^{d}}|h(\lambda,\xi)|^2\mu_0(d\lambda)\mu(d\xi)\bigg\}>0.
\label{eq: positivity imply no L^p norm 0}
\end{multline}
In order to go from \eqref{eq: positivity imply no L^p norm 0} to \eqref{eq: positivity imply no L^p norm}, we proceed as follows: recalling the definition of $\tilde{\mathcal{F}}$ in~\eqref{tilde cal F}, we will see in the next Lemma~\ref{F(g)} that if  $g\in\mathcal{A}_d$ then $\tilde{\mathcal{F}}g^2\in\overline{\mathcal{G}_0}$, where $\overline{\mathcal{G}_0}$ is defined by \eqref{bar g_0}. With the expression \eqref{eq: def of tilde h} for $\tilde{h}(s,x)$ in mind, a direct application of Parseval's identify yields
\begin{align}\label{term 4}
\int_0^1\int_{\R^d}\tilde{h}(s,x) \, g^2(s,x) \,dxds
=\int_{\R\times\R^{d}}h(\lambda,\xi) \, \tilde{\mathcal{F}}g^2(\lambda,\xi) \, \mu_0(d\lambda)\mu(d\xi).
\end{align}
\lastchange{Invoking the fact that $q-1=1/(p-1)$, we thus deduce that}
\begin{equation*}
\int_0^1\int_{\R^d}\tilde{h}(s,x)g^2(s,x)dxds-\frac{q-1}{2t^{2H_0-1}}\int_{\R\times\R^d}|h(\lambda,\xi)|^2\mu_0(d\lambda)\mu(d\xi)
=\lastchange{\phi_t}(h,g),
\end{equation*}
where $\lastchange{\phi_t}(h,g)$ is defined by \eqref{eq:def-phi-h-g}. Hence it is readily checked that condition \eqref{eq: positivity imply no L^p norm 0} is equivalent to \eqref{eq: positivity imply no L^p norm}, which finishes our proof.

\end{proof}

We now turn to the technical result used in the proof of Lemma \ref{lemma: positivity imply no L^p norm}.
\begin{lemma}\label{F(g)}
Recall that the space $\mathcal{A}_d$ is given in Definition \ref{def: A_d}. Let $\overline{\mathcal{G}_0}$ be the space defined by 
\begin{align}\label{bar g_0}\overline{\mathcal{G}_0}=\{h\in L^2(\R\times\R^d; \mu_0\otimes\mu); h(-\lambda,-\xi)=\bar{h}(\lambda,\xi)\ \ \mu_0\otimes\mu-\text{a.e.}\}\end{align}
 Then for any $g\in\mathcal{A}_d$, we have $\tilde{\mathcal{F}}(g^2)\in\overline{\mathcal{G}_0}.$
\end{lemma}
\begin{proof}
This is a direct consequence of inequality \eqref{eq:funct-ineq} and definition of $\tilde{\mathcal{F}}$ in \eqref{tilde cal F}.
%We proceed by contradiction. Namely, suppose there exists a $g_0\in\mathcal{A}_d$ such that $\mathcal{F}(g_0^2)\not\in\overline{\mathcal{G}_0}$. Then, since $\mathcal{G}_0$ is dense in $\overline{\mathcal{G}_0}$, there exists a sequence  $h_n\in\mathcal{G}_0$ such that \begin{align}\label{eq: h_n}\int_{\R^{d+1}}|h_n(\lambda,\xi)|^2\mu_0(d\lambda)\mu(d\xi)=1,\quad
%\text{and}\ \  \int_{\R^{d+1}}{h_n}(\lambda,\xi)\mathcal{F}(g_0^2)(\lambda,\xi)\mu_0(d\lambda)\mu(d\xi)\uparrow\infty.\end{align}
%In particular, fix a given $t<t_0(p)$.  Then for all large enough $n$, according to \eqref{eq: h_n},  the \lastchange{left-hand side} of \eqref{eq: positivity imply no L^p norm} is strictly positive.  Along the same lines as for Lemma \ref{lemma: positivity imply no L^p norm} we thus get
%$$\liminf_{\ep\to0}\|u^{\ep,\di}(t,x)\|_p=\infty.$$
%This contradicts the fact that $L^p(\Omega)-\lim_{\epsilon\to0}u_t^{\epsilon,\di}(x)=u_t(x)$ for $t<t_0(p)$, %which is shown in Theorem \ref{thm:explosion-critical}.
\end{proof}

Now we are ready to prove the main result of this section.

\begin{proof}[Proof of  Theorem \ref{thm:explosion-critical}  - item \ref{critical-iii}] We have shown that $\|u_t^{\epsilon,\di}(x)\|_p$ diverges as $\epsilon\to0$ as long as relation \eqref{eq: positivity imply no L^p norm} is satisfied. {Note that since $\tilde{\mathcal{F}}(g^2)\in\overline{\mathcal{G}_0}$, the functional $\phi_t(h,g)$ defined in~\eqref{eq:def-phi-h-g} is continuous in $h$ with respect to the $L^2(\R^{d+1}, \mu_0\otimes\mu)$-norm. Therefore,
$$\sup_{h\in\mathcal{G}_0}\phi_t(h,g)=\sup_{h\in\overline{\mathcal{G}_0}}\phi_t(h,g).$$
Hence, we aim to show that \eqref{eq: positivity imply no L^p norm} is satisfied for a suitable family of functions $h\in\overline{\mathcal{G}_0}$.}
Namely, we want to find some functions $h\in\overline{\mathcal{G}_0}$ such that
\begin{align}\label{eq: positivity imply no L^p norm 2}
\sup_{g\in\mathcal{A}_d}\left\{\lastchange{\phi_t}(h,g)-\frac{1}{2}\int_0^1\int_{\R^d}|\nabla_xg(s,x)|^2dxds\right\}>0.
\end{align}
To this end, we will consider a family of functions $\{h^{\theta,g}; \theta\in\R, g\in\mathcal{A}_d\}$. Each $h^{\theta,g}$ is defined by
$$h^{\theta,g}(\lambda,\xi)=\theta\big[\tilde{\mathcal{F}}(g^2)\big](-\lambda,-\xi)=\theta\lastchange{\overline{\big[\tilde{\mathcal{F}}(g^2)\big]}}(\lambda,\xi),$$
where we recall that $\tilde{\mathcal{F}}(g^2)\in\overline{\mathcal{G}_0}$ according to Lemma \ref{F(g)}.  For such a function $h^{\theta,g}$ we have
\[
\lastchange{\phi_t}(h^{\theta,g},g)=\left(\theta-\frac{\theta^2}{2t^{2H_0-1}(p-1)}\right)\int_{\R\times\R^d}|\tilde{\mathcal{F}}(g^2)(\lambda,\xi)|^2\mu_0(d\lambda)\mu(d\xi).
\]
Thus an elementary computation reveals that
\[
\sup_{\theta\in\R}\lastchange{\phi_t}(h^{\theta,g},g)=\frac{(p-1)t^{2H_0-1}}{2}\int_{\R\times\R^d}|\tilde{\mathcal{F}}(g^2)(\lambda,\xi)|^2\mu(d\lambda)\mu(d\xi).
\]
Plugging this information into \eqref{eq: positivity imply no L^p norm 2} we get 
\begin{align*}
&\sup_{h\in\overline{\mathcal{G}_0}}\sup_{g\in\mathcal{A}_d}\left\{\lastchange{\phi_t}(h,g)-\frac{1}{2}\int_0^1\int_{\R^d}|\nabla_xg(s,x)|^2dxds\right\}\\
&\geq\sup_{g\in\mathcal{A}_d}\Big\{\frac{(p-1)t^{2H_0-1}}{2}\int_{\R\times\R^d}|\tilde{\mathcal{F}}(g^2)(\lambda,\xi)|^2\mu(d\lambda)\mu(d\xi) -\frac{1}{2}\int_0^1\int_{\R^d}|\nabla_xg(s,x)|^2dxds \Big\}.
\end{align*}
According to Notation \ref{not:kappa} and to the definition of $t_0(p)$ in Theorem \ref{thm:explosion-critical}, it is clear that the above quantity is strictly positive as soon as $t>t_0(p)$. This concludes the proof. 
\end{proof}

\section{Stratonovich case}\label{sec:strato}

In this section we analyze the moments of  equation~\eqref{eq:she-intro} interpreted in the Stratonovich sense, for a wide class of noises. We first briefly recall, in Section \ref{sec:existence-strato}, the main wellposedness result of \cite{CDOT2} regarding the Stratonovich equation. The moment analysis will then take place in Section~\ref{section:moments-estim}.

\subsection{Interpretation of the solution}\label{sec:existence-strato}

The basic idea behind the Stratonovich interpretation of equation~\eqref{eq:she-intro} can be roughly expressed as follows. We consider a sequence $\{\dot{W}^n; n\geq1\}$ of smooth approximations of $\dot{W}$, which can be thought of as the mollification considered in \eqref{smoothed noise} with $\epsilon=1/n$. With this approximation in hand, let $\{u^n; n\geq1\}$ be the sequence of classical solutions associated with $\dot{W}^n$, that is $u^n$ is the solution of
\begin{equation}\label{eq:she-intro-approx}
\partial_{t} u^n_{t}(x) = \frac12 \Delta u^n_{t}(x) + u^n_{t}(x) \, \dot W^n_{t}(x), \quad t\in\R_{+}, x\in \R^{d},
\end{equation}
understood in the classical Lebesgue sense. Then the Stratonovich solution of \eqref{eq:she-intro} is morally defined as the limit (if it exists) of the sequence $u^n$.

In some situations where $\dot{W}$ is not too rough, the above heuristic idea can be rigourously formulated within the so-called Young setting (see e.g. \cite[Section 5]{HHNT}). In order to extend these considerations to rougher noises (which overall corresponds to our objective in this study), more intricate machineries must come into the picture, as well as renormalization procedures.

\smallskip

Thus, in the companion paper \cite{CDOT2}, we have relied on the sophisticated \textit{theory of regularity structures} (as introduced by Hairer in \cite{hai-14}) to provide a complete treatment of the Stratonovich model, i.e. to show existence and uniqueness of a global solution, in a rough regime. For the sake of conciseness, we will here skip the details of this analysis (which involve the exhibition of an abstract solution map and the construction of a related $K$-rough path), and will directly provide the resulting convergence statement for the approximated equation \eqref{eq:she-intro-approx}.

\smallskip

To this end, we shall need the following piece of notation.

\begin{notation}
Let $\rho$ be the weight given by $\rho(s,x):=p_1(s)p_1(x)$ as considered in \eqref{smoothed noise} and define the approximated noise $\dw^n$ of $\dot{W}$ by $\dw^0:=0$ and for $n\geq 1$, 
\begin{align}\label{def:approximate noise}
\dw^n:=\partial_t  \partial_{x_1} \cdots \partial_{x_d} W^n ,
\end{align}
where $W^n:=\rho_n \ast W$ and $\rho_n(s,x):=2^{n(d+2)}\rho(2^{2n} s,2^n x)$. 

\smallskip

For any vector $(H_0, \bh)\in (0,1)^{d+1}$, we set from now on
\begin{equation}\label{notation:cn-h}
\cn_{H_0,\bh}(\la,\xi):=\frac{1}{|\la|^{2H_0-1}} \, \prod_{i=1}^d \frac{1}{|\xi_i|^{2H_i-1}} \, ,
\end{equation}
namely $c_0c_\bh\,\cn_{H_0,\bh}$ is the density of the measure $\mu_0\otimes\mu$ introduced in \eqref{eq:def-mu}. Let us also set 
\begin{equation}\label{cstt-c-n}
c_{H_0,\bh}=\big(\prod_{i=0}^d \al_{H_i} \big)^{-1/2}  \ \text{with} \ \al_{H_i}:=\int_{\R} d\xi \, \frac{|e^{\imath \xi}-1|^2}{|\xi|^{2H_i+1}}.
\end{equation}
Finally, note that when $2H_0+H < d+1$ (where the notation $H$ has been introduced in \eqref{defi:j-ast}) it can be shown that the integral below is finite (see \cite{CDOT2}):
\begin{equation}\label{eq:cj-c-n}
\cj_{\rho,H_0,\bh}:=
\int_{\R^{d+1}}  |\mathcal F{\rho}(\la,\xi)|^2  \mathcal F{p}(\la,\xi) \cn_{H_0,\bh}(\la,\xi)\, d\la d\xi .
\end{equation}
\end{notation}
With this notation in hand, we are ready to recall the main result of \cite{CDOT2} about the approximated equation.
\begin{definition}\label{defi:strato-sol}
Let $(H_0, \bh)\in (0,1)^{d+1}$ be a vector of Hurst parameters such that 
\begin{equation}\label{cond-h-strato}
d+\frac23< 2H_0+H \leq d+1\, .% \quad \text{and} \quad |\tilde{J}_\ast| \leq 1 \, ,
\end{equation}
%where  $\tilde{J}_\ast:=\left\{j\in \{1,\ldots,d\}: \, H_i \leq \frac12\right\}.$
Fix $\al \in \R$ such that 
$$-\frac43 < \al < -(d+2)+2H_0+H \, ,$$
as well as an arbitrary time horizon $T>0$ and an initial condition $\psi\in L^\infty(\R^d)$. Finally, let $(\mathfrak{c}_{\rho,H_0,\bh}^{(n)})_{n\geq 1}$ be the sequence defined for every $n\geq 1$ as
\begin{align}
&\mathfrak{c}_{\rho,H_0,\bh}^{(n)}:=\nonumber\\
&
\begin{cases}
c_{H_0,\bh}^2 \, 2^{2n(d+1-(2H_0+H))} \cj_{\rho,H_0,\bh}   & \text{if} \ 2H_0+H<d+1\\
c_{H_0,\bh}^2 \int_{|\la|+|\xi|^2\geq 2^{-2n}}  |\mathcal F{\rho}(\la,\xi)|^2  \mathcal F{p}(\la,\xi) \cn_{H_0,\bh}(\la,\xi)\, d\la d\xi& \text{if} \ 2H_0+H=d+1
\end{cases}\label{renormal}
\end{align}
and consider the sequence $(u^{n})_{n\geq 1}$ of classical solutions of the equation 
\begin{equation}\label{eq:u-square-n}
\left\{
\begin{array}{l}
\partial_{t} u^{n}  = \frac12 \Delta u^{n} + u^{n}\, \dw^n-\mathfrak{c}_{\rho,H_0,\bh}^{(n)}\, u^{n} \, , \quad t\in [0,T],\,  x\in \R^{d} \, ,\\
u^{n}_0(x)=\psi(x) \ .
\end{array}
\right.
\end{equation}
Then the sequence $u^n$ converges almost surely in $L^\infty([0,T]\times \R^d)$. The limit $u=\lim{u^n}$ is said to be the Stratonovich solution to the renormalized equation of \eqref{eq:she-intro}.
\end{definition}

\color{black}

%\begin{remark}
%Observe that the first assumption $d+\frac23< 2H_0+H < d+1$ in \eqref{cond-h-strato} entails that $|\bj_\ast| \leq 2$, and so the second assumption in \eqref{cond-h-strato} actually reduces to $|\bj_\ast|\neq 2$. This last condition is still restrictive though, since we may have $d+\frac23< 2H_0+H < d+1$ and $|\bj_\ast|=2$ (take $d=2$, $H_0=1-\varepsilon,H_1=\frac12-\varepsilon,H_2=\frac12-\varepsilon$, for $\varepsilon >0$ small enough). 
%\hre{One can reach this regime thanks to some extra layers of regularity structure type expansions.}
%On the other hand, observe that the condition $H_1+\ldots+H_d >d-1$ (ensuring the existence of a global Skorohod solution) yields $|\bj_\ast| \leq 1$, as required in the above statement.
%\end{remark}

\subsection{Moments estimates for the Stratonovich equation}\label{section:moments-estim}

%\tbc{[Aurelien] Would you like me to \enquote{harmonize} (and improve parts of) the presentation of this section ? If so, then I would just need some help regarding my comment at the very end of the section, and also about the possibility to write the previous Skorohod results with the \enquote{usual} approximation by the mollifier $\vp_\delta \otimes p_\varepsilon$ (instead of the mixed Young-smooth approximation in \eqref{eq:def-Vt-epsilon-2}).}
%{\color{blue}[Cheng] If we use mollifier $\vp_\delta\otimes p_\epsilon$ we should set (in comparison with \eqref{eq:def-Vt-epsilon-2})
%$$V^{\epsilon,\delta, B}_t(x)=W\left(\int_0^t\vp_\delta(u-\cdot)p_\epsilon(B^x_{t-u}-\cdot)du\right),$$
%and
%\begin{align*}
%\beta^{\epsilon,\delta,B}_t=\E\big[|V_t^{\epsilon,\delta,B}(x)|^2\big]=\int_{[0,t]^2}ds_1ds_2\int_{\R^{d+1}}e^{i(s_1-s_2)\lambda+i(B^x_{t-s_1}-\tilde{B}^x_{t-s_2})}\mu_0^\delta(d\lambda)\mu^\epsilon(d\la).
%\end{align*}
%Here $\mu^\delta_0=\mathcal{F}(\vp_\delta)\cdot\mu_0$ and $\mu^\epsilon=\mathcal{F}(p_\epsilon)\cdot\mu$. In particular, if we choose $\vp_\delta$ and $p_\epsilon$ to be one and d dimensional Gaussian densities, we have
%$$\mu_0^\delta(d\lambda)=e^{-\frac{1}{2}\delta \lambda^2}\mu_0(d\lambda),\quad\text{and}\ \mu^\epsilon(d\la)=e^{-\frac{1}{2}\epsilon |\la|^2}\mu(d\la).$$
%Clearly, they are dominated by $\mu_0$ and $\mu$.
%}

Similarly to the Skorohod situation (see Section \ref{section:Skorohod case}), we now would like to show that the moments of the renormalized Stratonovich solution $u $ (as introduced in Definition \ref{defi:strato-sol}) are all finite in the subcritical regime~\eqref{eq:sko-subcritical-regime}. With this objective in mind, observe first that  Definition \ref{defi:strato-sol} only guarantees \emph{almost sure} convergence of $u^{n}$ to $u $, and does not provide any estimate on the moments of $u$. In order to go further, our strategy will somehow consist in bounding the moments of the Stratonovich solution $u$ in terms of those of the Skorohod solution $u^\diamond$ (at the level of their respective approximations), and then exploiting the estimates of Section \ref{section:Skorohod case} for the moments of $u^\diamond$. 

{For more clarity, and although both Proposition~\ref{prop:finite-moments-subcritical} and  Definition~\ref{defi:strato-sol} are valid for $H_{0}\le\frac12$ (see \cite{Ch19} for the Skorohod case), we will restrict our analysis to the case where $H_0>\frac12$ in the sequel. Besides, for our comparison strategy to be possible, we of course need to focus on situations where both $u $ and $u^\diamond$ are well defined. \textit{Therefore, for the remainder of the section, we will assume that both conditions \eqref{eq:sko-subcritical-regime} and \eqref{cond-h-strato} are met.}}
%It is easy to see, by combining the restrictions in \eqref{eq:sko-subcritical-regime} and \eqref{cond-h-strato}, that this intersection domain can be described through the three conditions
%\begin{equation}\label{inter-domain}
%H>d-1 \quad , \quad d+\frac23< 2H_0+H < d+1 \quad \text{and} \quad |\tilde{J}_\ast| \leq 1 \, ,
%\end{equation}
%where we recall that $\tilde{J}_\ast:=\big\{j\in \{1,\ldots,d\}: \, H_i \leq \frac12\big\}$. We assume that the condition \eqref{inter-domain} is met in the remainder of the section.

As mentioned above, the idea will be to compare $u^\diamond$ and $u $ through the Feynman-Kac representations of their respective approximations. Just as in Section \ref{section:Skorohod case} ({resp. Definition \ref{defi:strato-sol}}) we denote by $u^{n,\diamond}$ (resp. $u^{n}$) the approximation of $u^\diamond$ (resp. $u $). Namely $u^{n,\diamond}$ is the solution of \eqref{eq:heat-sko-approx}, while $u^{n}$ solves \eqref{eq:u-square-n}, for every fixed $n\geq 1$. Let us assume that both solutions start from an initial condition $\psi\in L^{\infty}(\R^{d})$. At this point, let us recall that according to \cite[Proposition 5.2]{HN} (and just as in Proposition~\ref{approx solution}), the Feynman-Kac representation of $u^{n,\diamond}$ is given by
\begin{equation}\label{fk-u-di}
u _{t}^{n,\di}(x)=\mathbb{E}  \Big[ \psi(B^x_t) \exp \Big(V^{n,W}_t(B^x)-\frac12 \beta^{n,B}_{t} \Big)\Big]  \, ,
\end{equation}
where $B^x$ is a standard $d$-dimensional Wiener process starting at $x$, independent of $W$, and where $V^{n,W}, \beta^n_t$ are given by
$$V^{n,W}_t(B^x):=\int_0^t du \, \dw^{n}(u,B^x_{t-u}),  
\quad \text{and} \quad 
\beta^{n,B}_{t}:=\mathbb{E}_W\big[| V^{n,W}_t(B^x)|^2\big] \, .$$
On the other hand, it is clear that for every fixed $n\geq 1$, equation \eqref{eq:u-square-n} is a standard linear parabolic equation, for which a classical Feynman-Kac formula can be applied. This  yields the expression
\begin{equation}\label{fk-u-s} 
u^{n}_t(x)=\mathbb{E}  \Big[ \psi(B^x_t) \exp \Big(V^{n,W}_t(B^x)-\mathfrak{c}_{\rho,H_0,\bh}^{(n)} \, t \Big)\Big] \, ,
\end{equation}
where $\mathfrak{c}^{(n)}_{\rho,H_0,\bh}$ stands for the renormalization constant in equation \eqref{eq:u-square-n}.

\

Based on the two representations \eqref{fk-u-di} and \eqref{fk-u-s}, we deduce that the desired comparison between $u^{n,\diamond}$ and $u^{n}$ morally reduces to a comparison between $\frac12 \beta^{n,B}_{t}$ and $\mathfrak{c}^{(n)}_{\rho,H_0,\bh}\, t$. The following uniform estimate will thus be an important step in the procedure:

\begin{proposition}\label{prop:r-n-convol}
Assume that $H_0>\frac12$ and {that the condition \eqref{cond-h-strato} is satisfied}. 
%Let $\mathfrak{c}^{(n)}_{\rho,H_0,\bh}$ be the constant given in Definition \ref{defi:strato-sol}.
Then for every fixed $t\geq 0$, it holds that
\begin{equation}\label{link-beta-cstt-convol-bis}
\sup_{n\geq 1} \bigg| \mathfrak{c}^{(n)}_{\rho,H_0,\bh} \,  t-\frac12 \mathbb{E} \big[\beta^{n,B}_{t}\big]\bigg|  <  \infty .
\end{equation}
\end{proposition}

\begin{proof}
Following the definition \eqref{renormal} of $\mathfrak{c}_{\rho,H_0,\bh}^{(n)}$, we have to treat the cases $2H_0+H<d+1$ and $2H_0+H=d+1$ separately.

\smallskip

\noindent
\underline{First case: $2H_0+H<d+1$.} Let us recall that in this situation, $\mathfrak{c}^{(n)}_{\rho,H_0,\bh}$ is defined as
\begin{eqnarray}\label{d1-bis}
\mathfrak{c}^{(n)}_{\rho,H_0,\bh}
&:=&
\mathfrak{c}_{\rho,H_0,\bh} \cdot 2^{2n (d+1-(2H_0+H))}  \notag\\
&=&
c_{H_0,\bh}^2 \int_{\R^{d+1}} d\la d\xi\, |\cf\rho_n(\la, \xi)|^2  \cf p(\la,\xi) \cn_{H_0,\bh}(\la,\xi) \, ,
\end{eqnarray}
where $\cn_{H_0,\bh}(\la,\xi)$ is the quantity defined in \eqref{notation:cn-h}.
Keeping this expression in mind, and using the covariance formula \eqref{eq:covariance-W-with-Fourier}, we can recast the expression \eqref{eq:def-beta-epsilon} for $\beta^{n,B}_{t}$ as
\begin{align*}
&\beta^{n,B}_{t}=\be\big[| V^{n,W}_t(B^x)|^2\big]\\
&= c_{H_0,\bh}^2 \int_{\R^{d+1}} d\la d\xi\, |\cf \rho_n(\la,\xi)|^2 \cn_{H_0,\bh}(\la,\xi)\int_{[0,t]^{2}}dudv\,  e^{\ii  \lp \lambda(v-u)+\xi  \cdot  (B_{u} - B_{v})\rp}  \, .
\end{align*}
Therefore invoking elementary symmetry and integration arguments, we get
\begin{align}
&\frac12 \mathbb{E} \big[\beta^{n,B}_{t}\big]\nonumber\\
&= c_{H_0,\bh}^2 \int_{\R^{d+1}} d\la d\xi\, |\cf \rho_n(\la,\xi)|^2 \cn_{H_0,\bh}(\la,\xi)\int_0^t du\int_0^u dv \, e^{-\imath \la (u-v)} e^{-\frac{|\xi|^2}{2}(u-v)}\label{expr-expect-beta-n}\\
&= c_{H_0,\bh}^2 \int_{\R^{d+1}} d\la d\xi\, |\cf\rho_n(\la,\xi)|^2 \cn_{H_0,\bh}(\la,\xi)\int_0^t dv \, (t-v) e^{-v(\frac{|\xi|^2}{2}+\imath \la)} \nonumber\\
&=  c_{H_0,\bh}^2 \int_{\R^{d+1}} d\la d\xi\, |\cf\rho_n(\la,\xi)|^2 \cn_{H_0,\bh}(\la,\xi) 
\left[ \frac{t}{\frac{|\xi|^2}{2}+\imath \la}-\frac{1}{\frac{|\xi|^2}{2}+\imath \la}\int_0^t dv \, e^{-v(\frac{|\xi|^2}{2}+\imath \la)}\right].\label{trick-beta-n}
\end{align}
Hence owing to the fact that the Fourier transform of the heat kernel $p$ satisfies $\mathcal{F}p(\lambda,\xi)=(|\xi|^2/2+\imath\lambda)^{-1}$, together with \eqref{d1-bis}, we get
\begin{align}\label{eq:Ebeta^n and c^n-bis}
\frac12 \mathbb{E} \big[\beta^{n,B}_{t}\big]
=
\mathfrak{c}^{(n)}_{\rho,H_0,\bh} \, t-r^n_t \, ,
\end{align}
with a constant $r^n_t$ defined by
\begin{align}\label{eq:r^n_t-bis}
r^n_t:=2c_{H_0,\bh}^2 \int_{\R^{d+1}} d\la d\xi\, |\cf\rho_n(\la,\xi)|^2 \frac{\cn_{H_0,\bh}(\la,\xi) }{|\xi|^2+2\imath \la}\int_0^t dv \, e^{-v(\frac{|\xi|^2}{2}+\imath \la)} \, .
\end{align}
Since $\sup_{n\geq 1}|\cf\rho_n(\la,\xi)|^2 \lesssim 1$, our claim~\eqref{link-beta-cstt-convol-bis} amounts to prove that
\begin{equation}\label{condi-domi-bis}
\ca_{t}\equiv
\int_{\R^{d+1}} d\la d\xi\ \frac{\cn_{H_0,\bh}(\la,\xi) }{\big||\xi|^2+2\imath \la\big|}\bigg|\int_0^t dv \, e^{-v(\frac{|\xi|^2}{2}+\imath \la)}\bigg|  <  \infty  .
\end{equation}
In order to bound $\ca_t$ one must decompose the domain $\R^{d+1}$ into a centered ball and some unbounded domains. 

\smallskip

Regarding integration over the ball $\{|\la|\leq 1, |\xi_1|\leq 1,\ldots,|\xi_d|\leq 1\}$, one has, for all $\tau_0,\tau_1,\ldots,\tau_d\in [0,1]$ such that $\sum_{i=0}^d\tau_i=1$,
\begin{align}
&\int_{\{|\la|\leq 1, |\xi_1|\leq 1,\ldots,|\xi_d|\leq 1\}} d\la d\xi\ \frac{\cn_{H_0,\bh}(\la,\xi) }{\big||\xi|^2+2\imath \la\big|}\bigg|\int_0^t dv \, e^{-v(\frac{|\xi|^2}{2}+\imath \la)}\bigg| \label{integra-on-bounded-1}\\
&\leq t\int_{\{|\la|\leq 1, |\xi_1|\leq 1,\ldots,|\xi_d|\leq 1\}} \ \frac{d\la d\xi}{\big||\xi|^2+2\imath \la\big|}\frac{1}{|\la|^{2H_0-1}}\prod_{i=1}^d \frac{1}{|\xi_i|^{2H_i-1}}\nonumber\\
&\leq t\bigg(\int_{|\la|\leq 1} \frac{d\la}{|\la|^{2H_0+\tau_0-1}}\bigg)\bigg(\prod_{i=1}^d \int_{|\xi_i|\leq 1}\frac{d\xi_i}{|\xi_i|^{2H_i+2\tau_i-1}}\bigg). \label{integra-on-bounded-2}
\end{align}
At this point, note that due to the condition $2H_0+H<d+1$, we can actually pick the $\tau_i$'s such that $2H_0+\tau_0-1<1$ and $2H_i+2\tau_i-1<1$ for $i=1,\ldots,d$. Indeed, summing the latter constraints, we get $4H_0+2H+2\sum_{i=0}^d \tau_i<6$, which, combined with the relation $\sum_{i=0}^d\tau_i=1$, brings us back to the condition $2H_0+H<d+1$. Choosing this set of parameters in \eqref{integra-on-bounded-2}, we can conclude that the integral in \eqref{integra-on-bounded-1} is finite.

\smallskip

As for the integral on the remaining unbounded domains, we will only focus here on the region $\{|\la|\geq 1, |\xi_1|\geq 1,\ldots,|\xi_d|\geq 1\}$ in the right hand side of \eqref{condi-domi-bis}. In this case, for all $\tau_0,\tau_1,\ldots,\tau_d \in [0,1]$ such that $\sum_{i=0}^d \tau_i =1$, and recalling again the expression \eqref{notation:cn-h} for $\cn_{H_0,\bh}$, we have
\begin{align}
&\int_{|\la|\geq 1, |\xi_1|\geq 1,\ldots,|\xi_d|\geq 1} d\la d\xi\ \frac{\cn_{H_0,\bh}(\la,\xi) }{\big||\xi|^2+2\imath \la\big|}\bigg|\int_0^t dv \, e^{-v(\frac{|\xi|^2}{2}+\imath \la)}\bigg|\label{integra-on-unbounded}\\
\lesssim& \int_{|\la|\geq 1, |\xi_1|\geq 1,\ldots,|\xi_d|\geq 1} d\la d\xi\, \frac{\cn_{H_0,\bh}(\la,\xi)}{\la^2+\xi_1^4+\ldots+\xi_d^4}\nonumber\\ 
\lesssim& \bigg( \int_{|\la|\geq 1} \frac{d\la}{|\la|^{2H_0+2\tau_0-1}}\bigg) \prod_{i=1}^d  \int_{|\xi_i|\geq 1} \frac{d\xi_i}{|\xi_i|^{\lastchange{2}H_i+4\tau_i-1}}\, .\label{bound:technical in Prop Ebeta^n and c^n-bis}
\end{align}
We can now pick the $\tau_i$'s such that $2H_0+2\tau_0>2$ and $2H_i+4\tau_i >2$ for $i=1,\ldots,d$, due to the fact that we have
$$2(2H_0+2\tau_0)+\sum_{i=1}^d (2H_i+4\tau_i)=2(2H_0+H)+4 > 4+2d \, ,$$
where the last inequality is ensured by \eqref{cond-h-strato}. Choosing this set of parameters in \eqref{bound:technical in Prop Ebeta^n and c^n-bis}, we obtain that the integral in \eqref{integra-on-unbounded} is finite. 

\smallskip

The estimate of the integral on the other domains can then clearly be done along similar arguments, which achieves the proof of \eqref{condi-domi-bis}.

\

\noindent
\underline{Second case: $2H_0+H=d+1$.} In this situation, $\mathfrak{c}^{(n)}_{\rho,H_0,\bh}$ is defined as
$$\mathfrak{c}^{(n)}_{\rho,H_0,\bh}:=c_{H_0,\bh}^2 \int_{|\la|+|\xi|^2\geq 2^{-2n}}  |\mathcal F{\rho}(\la,\xi)|^2  \mathcal F{p}(\la,\xi) \cn_{H_0,\bh}(\la,\xi)\, d\la d\xi$$
which, thanks to the relation $2H_0+H=d+1$, can also be written as
\begin{align*}%\label{d1-bis-border}
&\mathfrak{c}^{(n)}_{\rho,H_0,\bh}=c_{H_0,\bh}^2 \int_{|\la|+|\xi|^2\geq 1} d\la d\xi\, |\cf\rho_n(\la,\xi)|^2  \cf p(\la,\xi) \cn_{H_0,\bh}(\la,\xi) \, .
\end{align*}
Therefore, going back to the expression \eqref{expr-expect-beta-n} of $\frac12 \mathbb{E} \big[\beta^{n,B}_{t}\big]$, and performing the same trick as in \eqref{trick-beta-n}, we get that
\begin{align*}
&\mathfrak{c}^{(n)}_{\rho,H_0,\bh} \,  t-\frac12 \mathbb{E} \big[\beta^{n,B}_{t}\big]\\
&=-c_{H_0,\bh}^2 \int_{|\la|+|\xi|^2\leq 1} d\la d\xi\, |\cf \rho_n(\la,\xi)|^2 \cn_{H_0,\bh}(\la,\xi)\int_0^t du\int_0^u dv \, e^{-\imath \la (u-v)} e^{-\frac{|\xi|^2}{2}(u-v)}\\
&+c_{H_0,\bh}^2 \int_{|\la|+|\xi|^2\geq 1} d\la d\xi\, |\cf\rho_n(\la,\xi)|^2 \cn_{H_0,\bh}(\la,\xi) \left[ \frac{1}{\frac{|\xi|^2}{2}+\imath \la}\int_0^t dv \, e^{-v(\frac{|\xi|^2}{2}+\imath \la)}\right],
\end{align*}
and accordingly
\begin{align*}
&\sup_{n\geq 1}\bigg|\mathfrak{c}^{(n)}_{\rho,H_0,\bh} \,  t-\frac12 \mathbb{E} \big[\beta^{n,B}_{t}\big]\bigg|\lesssim t\int_{|\la|+|\xi|^2\leq 1} d\la d\xi\, \cn_{H_0,\bh}(\la,\xi)\\
&\hspace{4cm}+\int_{|\la|+|\xi|^2\geq 1} d\la d\xi\ \frac{\cn_{H_0,\bh}(\la,\xi) }{\big||\xi|^2+2\imath \la\big|}\bigg|\int_0^t dv \, e^{-v(\frac{|\xi|^2}{2}+\imath \la)}\bigg| .
\end{align*}
By the very definition of $\cn_{H_0,\bh}$, the first integral in the latter bound is clearly finite. As for the second integral, its finiteness can be easily derived from similar arguments to those used for \eqref{integra-on-unbounded}, which achieves the proof of our assertion.
\end{proof}

\

\color{black}

Based on relations \eqref{fk-u-di}, \eqref{fk-u-s} and \eqref{link-beta-cstt-convol-bis}, the fact that one can transfer subcritical estimates from the Skorohod to the Stratonovich equation can be explained in the following way: {under the subcritical condition \eqref{eq:sko-subcritical-regime}}, the fluctuations of the random variable $Y^n_t$ given by
\begin{align}Y^n_t:=\beta^{n,B}_{t}-\E \left[\beta^{n,B}_{t}\right] \, 
\label{notation:Y^n_t}\end{align}
are much smaller than the fluctuation of the random variable $V^{n,W}_t(B)$ featuring in \eqref{fk-u-di} and~\eqref{fk-u-s}. This assertion is quantified in the following lemma.
\begin{lemma}\label{lem: integrability of Y}
Assume that $H_0>\frac12$ and that {the conditions \eqref{eq:sko-subcritical-regime} and \eqref{cond-h-strato} are both satisfied}. For $n\ge 1$, let $Y^n_t$ be the random variables defined by \eqref{notation:Y^n_t}. Then for all $\theta\geq0$ and $t\geq 0$, we have 
\begin{align}\label{eq:integrability of Y}\sup_{n\geq 1}\E\left[e^{\theta Y^n_t}\right]<\infty.\end{align}
\end{lemma}
\begin{proof}
 We first fix $n\geq1$. Our proof is inspired by the computations in \cite[Theorem 4.6]{{HHNT}}, where the exponential integrability of a random variable similar to $Y^n_t$ is investigated under more restrictive assumptions on $H_1,\dots,H_d$. As in \cite{HHNT}, our analysis will be based on Le Gall's decomposition for $\beta^n_t$.  More specifically, we start the following construction: for $N\geq 1$ and $k=1,...,2^{N-1}$ 
 we set
$$
J_{N,k}=\left[\frac{(2k-2)t}{2^N},\frac{(2k-1)t}{2^N}\right),  I_{N,k}=\left[\frac{(2k-1)t}{2^N}, \frac{2kt}{2^N}\right), 
$$
and $A_{N,k}=J_{N,k}\times I_{N,k}$. 
Then we can decompose $\beta^{n,B}_{t}$ as
$$\beta_t^{n,B}
=
\sum_{N=1}^{\infty}\sum_{k=1}^{2^N-1}a_{N,k}^n,\ \text{where}\ a_{N,k}^n=\int_{A_{N,k}}\gamma_0^n(r-s)\gamma^n(B_r^x-B_s^x)drds,$$
where $\gamma^n_0$ (resp. $\gamma^n$) is the inverse Fourier transform of the measure \lastchange{$\mu_0^n :=\mu_0^{2^{-4n}}$ (resp. $\mu^n:=\mu^{2^{-2n}}$)} defined by \eqref{eq:def-mu-epsilon}. Furthermore, owing to the shape of $A_{N,k}$'s combined with the independence of the Brownian increments, one can write
\begin{align}\label{law of a_N^n}
a_{N,k}^n\stackrel{(d)}{=}a_N^n,\quad \text{with}\quad a_N^n=\int_{[0,t/2^N]^2}\gamma_0^n(t+s)\gamma^n(B_r-\tilde{B}_s)dsdr,
\end{align}
where $B$ and $\tilde{B}$ are two independent  Brownian motions starting from the origin. It can be shown, exactly along the same lines as in \cite[equation (4.13)]{HHNT}, that the exponential integrability of $Y_t^n$ amounts to the following estimate: for any arbitrarily small $\ep>0$ there exists a constant $C_{\ep}>0$ such that for all $m\ge 1$ we have
\begin{align}\label{integrability Y_t key}
\E[(a_N^n)^m]\leq C_{\ep}m!\left(\frac{C\ep}{2^N}\right)^{\alpha m},
\end{align}
with $\alpha>1/2$. In fact, we should point out that the exponential integrability in~\cite{HHNT} was obtained thanks to~\eqref{integrability Y_t key} with $\al=1$. However, by a closer examination of the proof in \cite{HHNT}, it is clear that the exponential integrability \eqref{eq:integrability of Y}  follows as long as  $\alpha>1/2$ in~\eqref{integrability Y_t key}. The remainder of our proof is thus devoted to justify~\eqref{integrability Y_t key}.

Recall from \eqref{eq:def-mu-epsilon}  that $\mu_0^n$ and $\mu^n$ are  dominated by $\mu_0$ and $\mu$ respectively, where $\mu_0$ and $\mu$ are introduced in \eqref{eq:def-mu}. Therefore for  any $\tau>0$ we have
\begin{align}
&\E \lc \left(\int_{[0,\tau]^2}\gamma_0^n(t+s)\gamma^n(B_r-\tilde{B}_s)dsdr\right)^m \rc
\nonumber\\
=&\int_{(\R^{d+1})^m}\mu_0^n(d\lambda)\mu^n(d\xi)
\left|\int_{[0,\tau]^m}\E\lc \prod_{k=1}^m e^{i(\lambda_ks_k+\xi_k \cdot B_{s_k})} \rc ds\right|^2\nonumber\\
\leq&\int_{(\R^{d+1})^m}\mu_0(d\lambda)\mu(d\xi)\left|\int_{[0,\tau]^m}
\E\lc \prod_{k=1}^m e^{i(\lambda_ks_k+\xi_k\cdot B_{s_k})} \rc ds\right|^2\nonumber\\
=&\E \left[ \lp \int_{[0,\tau]^2}\gamma_0(t+s)\gamma(B_r-\tilde{B}_s)dsdr \rp^{m}\right].\label{domination uniform in n}
\end{align}
On the other hand, by \cite[Inequality (3.1)]{Ch18}, we have
\begin{align}\label{Xia's sharp estimate on moments}
\E \lc \left(\int_{[0,\tau]^2}\gamma_0(t-s)\gamma(B_r-\tilde{B}_s)dsdr\right)^m \rc
\leq C^m(m!)^{d-H}{\tau}^{(2H_0+H-d)m},
\end{align}
where $C>0$ is a constant independent of $\tau>0$.  Now combining \eqref{law of a_N^n}, \eqref{domination uniform in n} and \eqref{Xia's sharp estimate on moments}, together with the trivial relation that $\gamma_0(t+s)\leq\gamma_0(t-s)$, we have for all $m\geq1$,
$$\E [a_N^n]^m\leq C^m(m!)^{d-H}\left(\frac{t}{2^N}\right)^{(2H_0+H-d)m}\, .$$
Note that the above control of moments of $a_N^n$ is uniform in $n$. Moreover, under our assumption $d-H<1$, we can conclude that for any $\epsilon>0$ the above becomes
$$\E [a_N^n]^m\leq C^m(m!)\left(\frac{\varepsilon\, t}{2^N}\right)^{(2H_0+H-d)m},$$
when $m$ is large enough.
Finally, observe that we assume $2H_0+H-d>\frac12$ {by \eqref{cond-h-strato}}. Condition \eqref{integrability Y_t key} is therefore satisfied, and the integrability of $Y_t^n$ (uniformly in $n$) follows. This finishes the proof of \eqref{eq:integrability of Y}.
\end{proof}

We are finally in a position to prove the main result of this section, that is the fact that the Stratonovich solution  admits finite moments of any order.
\begin{proposition}\label{prop:finite-moments-subcritical-strato}
Assume that $H_0>\frac12$ and that {the conditions \eqref{eq:sko-subcritical-regime} and \eqref{cond-h-strato} are both satisfied}. Let $u$ be the solution of the Stratonovich equation as given by Definition \eqref{defi:strato-sol}, with initial condition $\psi\equiv1$. Then for all $t\ge 0$ and $p\ge 1$ we have
\begin{equation*}
\be\lc |u _{t}(x)|^{p}  \rc = c_{p,t} < \infty.
\end{equation*}
\end{proposition}

\begin{remark}
The extension of Proposition \ref{prop:finite-moments-subcritical-strato} to a general initial condition $\psi\in L^{\infty}(\R^{d})$ is straightforward. Details are omitted here for sake of clarity.
\end{remark}

\begin{proof}[Proof of Proposition \ref{prop:finite-moments-subcritical-strato}]
Recall from \eqref{fk-u-di} and \eqref{fk-u-s} that
$$u^{n}_t(x)=\mathbb{E}  \Big[  \exp \Big(V^{n,W}_t(B^x)-\mathfrak{c}^{(n)}_{\rho,H_0,\bh}\, t \Big)\Big],\  \ 
\text{and}\ \ u _{t}^{n,\di}(x)=\mathbb{E}  \Big[  \exp \Big(V^{n,W}_t(B^x)-\frac12 \beta^{n,B}_{t} \Big)\Big]   \, .$$
Without loss of generality, we assume that $p$ is an integer in the rest of the proof.
We also write $\beta_t^{n}(B^x)$ for $\beta^{n,B}_{t}$ in order to keep our notation visible enough.
We therefore get
\begin{align} \label{eq:Stra by Sko}
&u^{n}_t(x)=\E \left[ e^{V_t^{n,W}(B^x)-\frac12\beta_t^{n}(B^x)} \cdot e^{\frac12(\beta_t^{n}(B^x)-\E[\beta_t^{n}(B^x)])} \right] \cdot e^{\frac12\E[\beta_t^{n}(B^x)]-\mathfrak{c}^{(n)}_{\rho,H_0,\bh}\, t} \, .
%&=\E \left[\psi(B_t^x)e^{V_t^{n,W}(B^x)-\frac12\beta_t^{n,B}}\cdot e^{\frac12\beta_t^{n,B}-\frac12\E[\beta_t^{n,B}]}\cdot e^{\frac12\E[\beta_t^{n,B}]-\mathfrak{c}^{(n)}_{\rho,H_0,\bh} t}\right]\\
%&=e^{C_{H_0,H_1}t^{2H_0+H_1-1}}\E\left[e^{X_t}e^{Y_t}\right].
\end{align}
Taking expectation in \eqref{eq:Stra by Sko}, we end up with
\begin{multline}\label{eq:pth moment u Stra}
\be \big[\left(u^n_t(x)\right)^p\big]
=
\be\,\E \left[ \prod_{k=1}^pe^{V_t^{n,W}(B^{k,x})-\frac12\beta_t^{n}(B^{k,x})} \cdot \prod_{k=1}^pe^{\frac12(\beta_t^{n}(B^{k,x})-\E[\beta_t^{n}(B^{k,x})])} \right] \\
\times
 \prod_{k=1}^pe^{\frac12\E[\beta_t^{n}(B^{k,x})]-\mathfrak{c}^{(n)}_{\rho,H_0,\bh}\, t},
\end{multline}
where $\{B^{k,x}; k=1,\dots,p\}$ is a family of independent Brownian motions which is also independent from the noise $W$. Moreover, a standard application of Fubini's theorem in \eqref{eq:pth moment u Stra} yields
\begin{multline}
\be \big[\left(u^n_t(x)\right)^p\big]
=
\E\Bigg[ \be \left[\prod_{k=1}^pe^{V_t^{n,W}(B^{k,x})-\frac12\beta_t^{n}(B^{k,x})} \right] \\
\times 
\prod_{k=1}^pe^{\frac12(\beta_t^{n}(B^{k,x})-\E[\beta_t^{n}(B^{k,x})])} \Bigg]
\ 
 \prod_{k=1}^pe^{\frac12\E[\beta_t^{n}(B^{k,x})]-\,\mathfrak{c}^{(n)}_{\rho,H_0,\bh}\, t}.
 \label{eq:representation p moment u^n}
\end{multline}
Now, notice that
\begin{align}\label{eq: expectation with respect to W}\be\left[ \prod_{k=1}^pe^{V_t^{n,W}(B^{k,x})-\frac12\beta_t^{n}(B^{k,x})}\right]=\exp\left(\sum_{1\leq j_1<j_2\leq p}\alpha^{j_1j_2;n}_t\right),
\end{align}
where $\alpha^{j_1,j_2;n}_t$ is defined the same way as in \eqref{eq:def-alpha-j1j2} but with $\gamma, \gamma_0$ \lastchange{replaced} by $\gamma^n, \gamma_0^n$.  Plugging~\eqref{eq: expectation with respect to W} into \eqref{eq:representation p moment u^n}, we have
\begin{align*}
&\be\big[ \left(u^n_t(x)\right)^p\big]=
\E\left[ e^{\sum_{1\leq j_1<j_2\leq p}\alpha_t^{j_1j_2;n}}\, \prod_{k=1}^pe^{\frac12(\beta_t^{n}(B^{k,x})-\E[\beta_t^{n}(B^{k,x})])} \right] 
 \prod_{k=1}^pe^{\frac12\E[\beta_t^{n}(B^{k,x})]-\mathfrak{c}^{(n)}_{\rho,H_0,\bh}\, t}.
\end{align*}
On the other hand, invoking similar considerations as for \eqref{domination uniform in n}, we get that for any $\lambda>0$
\begin{align*}
\mathbb{E}\lc \exp\left(\lambda\alpha^{j_1j_2;n}_t\right) \rc
\leq
\mathbb{E}\lc\exp\left(\lambda\alpha^{j_1j_2}_t\right) \rc.
\end{align*}
The rest of the proof then follows from Proposition \ref{prop:FK-representation},  Proposition \ref{prop:r-n-convol}, Lemma \ref{lem: integrability of Y} and a simple application of Fatou's lemma and H\"{o}lder's inequality. 
\end{proof}


\begin{thebibliography}{}

\bibitem{AC2015}R. Allez and K. Chouk: The continuous Anderson Hamiltonian in dimension two. {\it arXiv:1511.02718}, (2015).


\bibitem{BCR09}
R. Bass, X. Chen and J. Rosen: Large deviations for Riesz potential of additive processes. {\it Ann. Inst. Henri Poincar\'{e} Probab. Stat.} {\bf 45} (2009), no. 3, 626-666.



\bibitem{CCGT}
P. Chakraborty,  X. Chen, B. Go and S. Tindel:
Quenched asymptotics for a 1-d stochastic heat equation driven by a rough spatial noise.
To appear in {\it Stoch. Proc. Appl.}



\bibitem{CH}
A. Chandra and M. Hairer:
An analytic BPHZ theorem for regularity structures.
{\it Arxiv preprint} (2016).



\bibitem{Ch-bk}
X. Chen: \emph{Random Walk Intersections: Large Deviations and Related Topics.}
American Mathematical Society. (2008)



\bibitem{Ch12}
X. Chen: Quenched asymptotics for Brownian motion of renormalized Poisson potential and for the related parabolic Anderson models.
{\it Ann. Probab.} {\bf 40} (2012), no. 4, 1436-1482.



\bibitem{Ch14}
X. Chen: Quenched asymptotics for Brownian motion in generalized Gaussian potential.
{\it Ann. Probab.} {\bf 42} (2014), no. 2, 576-622.



\bibitem{Ch17}
X. Chen: Moment asymptotics for parabolic Anderson equation with fractional time-space noise in Skorokhod regime. 
{\it Ann. Inst. Henri Poincar\'e Probab. Stat.} {\bf 53} (2017), no. 2, 819-841.



\bibitem{Ch18}
X. Chen: Parabolic Anderson model with rough or critical Gaussian noise.
{\it Ann. Institut Henri Poincar\'e Probab. Stat.} {\bf 55} (2019), no. 2, 941-976.  



\bibitem{Ch19}
X. Chen: Parabolic Anderson model with a fractional Gaussian noise that is rough in time. 
{\it Ann. Institut Henri Poincar\'e Probab. Stat.} {\bf 56} (2020), no. 2, 792-825.

%\bibitem{CDOT1}
%X. Chen, A. Deya, C. Ouyang and S. Tindel:
%Moment estimates for some renormalized parabolic Anderson models.
%{\it arXiv:2003.14367} (2020). \aurel{$\longrightarrow$ \textbf{Let us get rid of this intermediate Arxiv version.}}

\bibitem{CDOT2}
X. Chen, A. Deya, C. Ouyang and S. Tindel:
A $K$-rough path above the space-time fractional Brownian motion.
In preparation.

\bibitem{CHNT}
X. Chen, Y. Hu, D. Nualart and S. Tindel:
Spatial asymptotics for the parabolic Anderson model driven by a Gaussian rough noise. 
{\it Electron. J. Probab.} {\bf 22} (2017).



\bibitem{CHSX} 
X. Chen, Y. Hu, J. Song and F. Xing:
Exponential asymptotics for time-space Hamiltonians.
{\it Ann. Institut Henri Poincar\'e Probab. Stat.} {\bf 51} (2015), 1529-1561.



\bibitem{CFK}
D. Conus and D. Khoshnevisan: 
On the existence and position of the farthest peaks of a family of stochastic heat and wave equations. 
{\it Probab. Theory Related Fields} {\bf 152} (2012), no. 3-4, 681-701.



\bibitem{CFJK}
D. Conus, M. Foondun, M. Joseph and D. Khoshnevisan:
On the chaotic character of the stochastic heat equation II.
\textit{Probab. Theory Related Fields} {\bf 156} (2013), no. 3-4, 483-533.



\bibitem{CJK}
D. Conus, M. Joseph, and D. Khoshnevisan: 
On the chaotic character of the stochastic heat equation, before the onset of intermittency.
\textit{Ann. Probab.} \textbf{41} (2013), no. 3B, 2225-2260.



\bibitem{De16}
A. Deya: 
On a modelled rough heat equation.
{\it Probab. Theory Relat. Fields} {\bf 166} (2016), 1-65.



\bibitem{De17}
A. Deya: 
Construction and Shorohod representation of a fractional $K$-rough path. 
{\it Electron. J. Probab.} {\bf 22} (2017).



\bibitem{DZ}
A. Dembo and O. Zeitouni: 
{\it Large deviations techniques and applications}. Second edition. Applications of Mathematics (New York), {\bf 38}. Springer-Verlag, 1998. 



\bibitem{GH}
Y. Gu and W. Xu: 
Moments of 2D parabolic Anderson model. 
{\it Asymptot. Anal.} {\bf 108} (2018), no. 3, 151-161.



\bibitem{hai-14}
M. Hairer: A theory of regularity structures. {\it Invent. Math.} {\bf 198} (2014), no. 2, 269-504.



\bibitem{HL}
M. Hairer and C. Labb{\'e}: 
Multiplicative stochastic heat equations on the whole space. 
{\it J. Eur. Math. Soc.}  {\bf 20} (2018), no. 4, 1005-1054.



\bibitem{HKP}
T. Hida, H-H. Kuo, J. Potthoff and L. Streit:
{\it White noise. An infinite-dimensional calculus.} 
Kluwer Academic Publishers, 1993.



\bibitem{HHNT}
Y. Hu, J. Huang, D. Nualart and S. Tindel: Stochastic Heat Equations with General Multiplicative Gaussian Noises: H\"older Continuity and Intermittency.
{\it Electron. J. Probab.} {\bf 20} (2015), no. 55, 1-50.



\bibitem{HN} 
Y. Hu and D. Nualart: 
Stochastic heat equation driven by fractional noise and local time.  
\textit{Probab. Theory Related Fields} \textbf{143} (2009), no. 1-2, 285-328.

\bibitem{HLN}
J. Huang, K. L\^e, D. Nualart:
Large time asymptotics for the parabolic Anderson model driven by space and time correlated noise. 
{\it Stoch. Partial Differ. Equ. Anal. Comput.} {\bf 5} (2017), no. 4, 614-651.



\bibitem{Kh}
D. Khoshnevisan: Analysis of stochastic partial differential equations. 
CBMS Regional Conference Series in Mathematics, 119. American Mathematical Society, 2014.



\bibitem{konig_book}
W. K\"onig: \emph{The Parabolic Anderson Model: Random Walk in Random Potential.}
Birkh\"auser (2016).



\bibitem{Le}
K. L\^e:
A remark on a result of Xia Chen. 
{\it Statistics \& Probability Letters} {\bf 118} (2016), 124-126.



\bibitem{Nu-bk}
D. Nualart : {\it The Malliavin Calculus and Related Topics.} 
Second edition. Probability and its Applications (New York). Springer-Verlag, Berlin, 2006.


\bibitem{Nu-Za}
D. Nualart and  M. Zakai: Generalized multiple stochastic integrals and the representation of Wiener functionals. 
\textit{Stochastics} \textbf{23} (1988), 311-330.




\end{thebibliography}
\end{document}